\documentclass[11pt]{amsart}

\usepackage[a4paper,hmargin=3.5cm,vmargin=4cm]{geometry}
\usepackage{amsfonts,amssymb,amscd,amstext}

\usepackage{fancyhdr}
\usepackage[utf8]{inputenc}
\usepackage{hyperref}
\usepackage{verbatim}

\usepackage{times}

\usepackage{enumerate}
\usepackage{titlesec}
\usepackage{mathrsfs}

\titleformat{\section}
{\filcenter\bfseries\large} {\thesection{.}}{0.2cm}{}
\titleformat{\subsection}[runin]
{\bfseries} {\thesubsection{.}}{0.15cm}{}[.]
\titleformat{\subsubsection}[runin]
{\em}{\thesubsubsection{.}}{0.15cm}{}[.]

\usepackage[up,bf]{caption}

\newtheorem{theorem}{Theorem}[section]
\newtheorem{proposition}[theorem]{Proposition}

\newtheorem{lemma}[theorem]{Lemma}
\newtheorem{corollary}[theorem]{Corollary}

\theoremstyle{definition}
\newtheorem{definition}[theorem]{Definition}
\newtheorem{remark}[theorem]{Remark}

\newtheorem{problem}[theorem]{Problem}
\newtheorem{example}[theorem]{Example}

\numberwithin{equation}{section}
\numberwithin{figure}{section}

\newcommand\Ecal{\mathcal{E}}

\newcommand\Pcal{\mathcal{P}}

\newcommand\Rcal{\mathcal{R}}

\newcommand\Cscr{\mathscr{C}}

\newcommand\Oscr{\mathscr{O}}

\newcommand\C{\mathbb{C}}
\newcommand\D{\overline{\mathbb D}}

\renewcommand\D{\mathbb D}

\newcommand\N{\mathbb{N}}

\newcommand\R{\mathbb{R}}

\newcommand\Z{\mathbb{Z}}

\newcommand\igot{\mathfrak{i}}

\renewcommand\igot{\mathfrak{i}}

\newcommand\E{\mathrm{e}}

\renewcommand\imath{\igot}

\newcommand\hra{\hookrightarrow}
\newcommand\lra{\longrightarrow}

\newcommand\longhookrightarrow{\ensuremath{\lhook\joinrel\relbar\joinrel\rightarrow}}

\newcommand\wt{\widetilde}
\newcommand\wh{\widehat}
\newcommand\di{\partial}
\newcommand\dibar{\overline\partial}

\newcommand\Id{\mathrm{Id}}

\newcommand\Contf{\mathrm{Cont}_{\mathrm{for}}}

\newcommand\Conth{\mathrm{Cont}_{\mathrm{hol}}}
\newcommand\AH{\mathrm{AH}}
\newcommand\AC{\mathrm{AC}}

\numberwithin{equation}{section}

\begin{document}

\fancyhead[LO]{H-principle for complex contact structures on Stein manifolds}
\fancyhead[RE]{F.\ Forstneri{\v c}} 
\fancyhead[RO,LE]{\thepage}

\thispagestyle{empty}

\vspace*{1cm}
\begin{center}
{\bf\LARGE H-principle for complex contact structures \\ on Stein manifolds}

\vspace*{0.5cm}

{\large\bf  Franc Forstneri{\v c}} 
\end{center}

\vspace*{1cm}

\begin{quote}
{\small
\noindent {\bf Abstract}\hspace*{0.1cm}
In this paper we introduce the notion of a formal complex contact structure on an odd
dimensional complex manifold. Our main result is that every formal complex contact structure on a 
Stein manifold $X$ is homotopic to a holomorphic contact structure on a Stein domain 
$\Omega\subset X$ which is diffeotopic to $X$. 
We also prove a parametric h-principle in this setting, analogous to Gromov's h-principle 
for contact structures on smooth open manifolds. 
On Stein threefolds we obtain a complete homotopy classification of 
formal complex contact structures. Our method furnishes a parametric 
h-principle for germs of holomorphic contact structures along
totally real submanifolds of class $\Cscr^2$ in any complex manifold.

\vspace*{0.2cm}

\noindent{\bf Keywords}\hspace*{0.1cm}  Stein manifold, complex contact structure, h-principle

\vspace*{0.1cm}

\noindent{\bf MSC (2010)}\hspace*{0.1cm} 37J55; 53D10, 32E10, 32E30}

\end{quote}

\section{Introduction}
A {\em complex contact manifold} is a pair $(X,\xi)$, where $X$ is a complex manifold of 
(necessarily) odd dimension $2n+1\ge 3$ and $\xi$ is a completely nonintegrable
holomorphic hyperplane subbundle (a {\em contact subbundle}) of the holomorphic 
tangent bundle $TX$, meaning that the O'Neill tensor
$\xi\times \xi \to TX/\xi$, $(v,w)\mapsto [v,w] \!\!\mod \xi$, is nondegenerate. 
Note  that $\xi=\ker\alpha$ where $\alpha$  
is a holomorphic $1$-form on $X$ with values in the holomorphic line bundle $L=TX/\xi$ 
(the {\em normal bundle} of $\xi$) which realises the quotient projection
\begin{equation}\label{eq:alpha} 
	0 \lra \xi \longhookrightarrow TX \stackrel{\alpha}{\lra} L \lra 0.
\end{equation}
Thus, $\alpha$ is a holomorphic section of the twisted cotangent bundle $T^*X\otimes L$.
The contact condition is equivalent to $\alpha\wedge (d\alpha)^n\ne 0$ at every point of $X$.
A theorem of Darboux \cite{Darboux1882CRAS} says that $\xi$ is locally at any point
holomorphically contactomorphic to the standard contact bundle 
$\xi_{\mathrm{std}}=\ker \alpha_{\mathrm{std}}$ on $\C^{2n+1}$ given by the $1$-form 
$\alpha_{\mathrm{std}} = dz+\sum_{j=1}^n x_j dy_j$, 
where $(x,y,z)$ are complex coordinates on $\C^{2n+1}$.
(See also \cite{Moser1965TAMS} or \cite[p.~67]{Geiges2008} for the real case
and \cite[Theorem A.2]{AlarconForstnericLopez2017CM} for the holomorphic case.)

We denote by $\Conth(X)$ the space of all holomorphic contact forms on $X$
endowed with the compact-open topology. 

In this paper we study the existence and homotopy classification of complex contact structures 
on Stein manifolds of any dimension $2n+1\ge 3$. 
We begin by recalling a few general observations due to LeBrun and Salamon 
\cite{LeBrun1995IJM,LeBrunSalamon1994IM} which pertain to an arbitrary complex
contact manifold. 

If $\alpha\in \Conth(X)$ and $L=TX/\ker\alpha$, 
then $\omega=\alpha\wedge (d\alpha)^n$ is a holomorphic $(2n+1)$-form on $X$ 
with values in the line bundle $L^{n+1} = L^{\otimes (n+1)}$, i.e., an element of 
$H^0(X,K_X\otimes L^{n+1})$ where $K_X=\Lambda^{2n+1}T^*X$ is the canonical bundle 
of $X$. Being nowhere vanishing, $\omega$ defines a holomorphic trivialisation of the line bundle 
$K_X\otimes L^{n+1}$, so we conclude that 
\begin{equation}\label{eq:KX}
	K_X^{-1} = K^*_X\cong L^{n+1}.
\end{equation}
Similarly, $(d\alpha)^n|_\xi$ is a nowhere vanishing section of the line bundle
$(\Lambda^{2n} \xi)^*\otimes L^{n}$ (i.e., $d\alpha|_\xi$ is an $L$-valued 
complex symplectic form on the bundle $\xi$), so we have that
\begin{equation}\label{eq:Lambda}
	\Lambda^{2n} \xi \cong L^{n} = (TX/\xi)^{n}.
\end{equation} 
In particular, on a contact $3$-fold we have $\Lambda^2\xi \cong TX/\xi$.
It is easily seen that conditions \eqref{eq:KX} and \eqref{eq:Lambda} are equivalent
to each other. These facts impose strong restrictions on the existence of complex contact 
structures, especially on compact manifolds. In particular, 
if $X$ is compact and simply connected, it carries at most one complex
contact structure up to isotopy (see \cite[Proposition 2.2]{LeBrunSalamon1994IM}).
For further results and references we refer to the survey by Beauville \cite{Beauville2011} 
and the introduction to the paper \cite{AlarconForstneric2017IMRN} by Alarc\'on and the author.

Assume now that $X$ is a Stein manifold of dimension $2n+1\ge 3$. 
For a generic holomorphic $1$-form
$\alpha$ on $X$, the equation $\alpha\wedge (d\alpha)^n = 0$ defines a (possibly empty)
complex hypersurface $\Sigma_\alpha\subset X$, and $\alpha$ is a contact form 
on the Stein manifold $X\setminus \Sigma_\alpha$.
This observation shows that there exist a plethora of Stein contact manifolds,
but does not answer the question whether a given Stein manifold (or a  given
diffeomorphism class of Stein manifolds) admits a contact structure. More precisely, 
{\em when is a complex hyperplane subbundle $\xi\subset TX$ satisfying 
\eqref{eq:Lambda} homotopic to a holomorphic contact subbundle?}

The following notion is motivated by Gromov's h-principle for real contact structures on smooth open 
manifolds (see \cite{Gromov1969} or \cite[10.3.2]{EliashbergMishachev2002}).

%
%
\begin{definition}[Formal complex contact structure]\label{def:formal}
Let $X$ be a complex manifold of dimension $2n+1\ge 3$. 
A  {\em formal complex contact structure} on $X$ is a  pair $(\alpha,\beta)$, 
where $\alpha$ is a smooth $(1,0)$-form on $X$ with values in a complex line
bundle $L\to X$ satisfying  \eqref{eq:KX}, 
$\beta$ is a smooth $(2,0)$-form on $\xi=\ker\alpha$ with values in $L$, and 
\begin{equation} \label{eq:formal}
	\alpha\wedge \beta^n =
	 \alpha \wedge \overbrace{\beta\wedge\cdots\wedge\beta}^n \ne 0
	\quad\text{holds at every point of $X$}.
\end{equation} 
\end{definition}

Note that $\alpha$ is a nowhere vanishing section of the complex 
vector bundle $T^*X\otimes L$ of rank $\dim X$; such always exists if $X$ is a Stein manifold.
A $(2,0)$-form $\beta$ satisfying \eqref{eq:formal} is an $L$-valued complex symplectic form 
on the complex $2n$-plane bundle $\xi=\ker\alpha \subset TX$,
and $\alpha\wedge \beta^n$ is a topological trivialisation of $K_X\otimes L^{n+1}$. 
On a Stein manifold, every complex vector bundle carries  a unique structure of
a holomorphic vector bundle up to isomorphisms according to the Oka-Grauert principle
(see \cite[Theorem 5.3.1]{Forstneric2017E}).

We denote by $\Contf(X)$ the space of all formal complex contact structures on $X$
endowed with the $\Cscr^\infty$ compact-open topology. We have the natural inclusion
\begin{equation}\label{eq:Cont}
	\Conth(X) \, \longhookrightarrow \, \Contf(X),\qquad 
	\alpha\mapsto (\alpha,d\alpha|_{\ker\alpha}).
\end{equation}

The following is our first main result; it is proved in Sect.\ \ref{sec:proofs}. 

%
%
\begin{theorem}\label{th:basic}
Let $X$ be a Stein manifold. 
Given $(\alpha_0,\beta_0)\in \Contf(X)$, there are a Stein domain $\Omega\subset X$
diffeotopic to $X$ and a homotopy $(\alpha_t,\beta_t) \in  \Contf(X)$ 
$(t\in [0,1])$ such that $\alpha_1|_\Omega\in\Conth(\Omega)$ and 
$\beta_1|_{\ker\alpha_1}=d\alpha_1|_{\ker\alpha_1}$ on $\Omega$. Furthermore, if  
$\alpha_0,\alpha_1\in\Conth(X)$ are connected by a path in $\Contf(X)$, they are 
also connected by a path of holomorphic contact forms on some Stein domain 
$\Omega\subset X$ diffeotopic to $X$.
\end{theorem}
 
A domain $\Omega\subset X$ is said to be {\em diffeotopic} to $X$ if there is a smooth family
of diffeomorphisms $h_t:X \stackrel{\cong}{\lra} h_t(X)\subset X$ $(t\in [0,1])$ such that 
$h_0=\Id_X$ and $h_1(X)=\Omega$. If $J$ denotes the almost complex structure operator on $X$, 
then $J_t=h_t^*(J)$ is a homotopy of complex structures on $X$ with $J_0=J$ and 
$J_1=h_1^*(J|_{\Omega})$. 
By Cieliebak and Eliashberg \cite[Theorem 8.43 and Remark 8.44]{CieliebakEliashberg2012}
the domain $\Omega$, and the diffeotopy $\{h_t\}_{t\in [0,1]}$ in Theorem \ref{th:basic} can be 
chosen such that  for every $t\in[0,1]$ the domain $h_t(X) \subset X$ is Stein;
equivalently, the manifold $(X,J_t)$ with $J_t=h_t^*(J)$ is Stein.

%
%
\begin{remark}
In our definition of a formal complex contact structure we may (and often do)
consider $(2,0)$-forms $\beta$ defined on $TX$, and not merely on the subbundle 
$\ker\alpha$; however, only the restriction $\beta|_{\ker\alpha}$ contributes to the product
$\alpha\wedge\beta$. On the other hand, the differential
of a holomorphic $L$-valued $1$-form $\alpha$ on $X$ is not an 
$L$-valued $2$-form on $X$ in general if $L$ is nontrivial, the reason being that 
for any holomorphic function $f$ we have that $d(f\alpha) = fd\alpha + df\wedge\alpha$. 
This shows however that the restriction
$d\alpha|_{\ker\alpha}$ is a well defined $L$-valued $2$-form on the subbundle $\ker\alpha$,
and $\alpha\wedge (d\alpha)^k$ is an $L^{k+1}$-valued form on $X$ for any $k\in\N$.
When writing $d\alpha$ for a holomorphic $1$-form $\alpha$ with values
in a nontrivial line bundle $L\to X$, we shall always mean $d\alpha|_{\ker\alpha}$,
and the equation $\beta=d\alpha$ will be understood to hold modulo $\alpha$
(i.e., on the subbundle $\ker\alpha$). 
\end{remark}

We also prove a parametric version of Theorem \ref{th:basic} (see Theorem \ref{th:parametric}) 
which says that a continuous compact family of formal complex contact structures on $X$ can be 
deformed to a continuous family of holomorphic contact structures on a Stein domain 
$\Omega\subset X$ diffeotopic to $X$, and  the deformation may be kept fixed
for those values of the parameter for which the given formal structure is 
already a holomorphic contact structure. 

For real contact structures, Gromov's h-principle 
\cite{Gromov1969} says that the inclusion \eqref{eq:Cont} 
of the space of smooth contact forms into the space of formal contact forms 
is a weak homotopy equivalence on any smooth open manifold;
in particular, every formal contact structure is homotopic to an honest contact structure.
(See also Eliashberg and Mishachev \cite[Sect. 10.3]{EliashbergMishachev2002}.) 
The situation is more complicated for closed manifolds as was discovered later 
by Bennequin \cite{Bennequin1983AST}
and Eliashberg \cite{Eliashberg1989IM,Eliashberg1993IMRN}. In particular, the h-principle 
for real contact structures fails on the $3$-sphere, but it holds for the 
class of overtwisted contact structures on any compact orientable  $3$-manifold;
see \cite[Theorem 1.6.1]{Eliashberg1989IM}. This was extended to manifolds of 
dimensions $\ge 5$ by Borman, Eliashberg, and Murphy in 2015; see
\cite{BormanEliashbergMurphy2015}.

Our results in the present paper seem to be the first analogues  in the holomorphic category
of the above mentioned Gromov's h-principle.
At this time we are unable to construct holomorphic contact forms on the whole 
Stein manifold under consideration. The main, and seemingly highly 
nontrivial problem arising in the proof, is the following.
(The analogous approximation problem for integrable holomorphic subbundles  
--- holomorphic foliations --- is also open in general; see \cite[Problem 9.16.8]{Forstneric2017E}.)

%
%
\begin{problem}\label{pr:Runge}
Given a holomorphic contact form $\alpha$ on an open neighbourhood of 
a compact convex set $K\subset \C^{2n+1}$, can we approximate $\alpha$ 
uniformly on $K$ by holomorphic contact forms defined on $\C^{2n+1}$? 
Is such approximation also possible for any continuous family of holomorphic contact forms 
$\alpha_p$ with parameter $p\in P$ in a compact Hausdorff space?
\end{problem}

This issue does not appear in the smooth case since one can 
pull back a contact structure on a neighbourhood $U$ of a compact convex set 
$K\subset \R^{2n+1}$ to a contact structure on 
$\R^{2n+1}$ by a diffeomorphism $\R^{2n+1}\to U$ which equals the identity near $K$.
%
%

%
%
\begin{theorem}\label{th:ifthen}
If Problem \ref{pr:Runge} has an affirmative answer, then every formal complex contact
structure on a Stein manifold $X$ is homotopic to a holomorphic contact
structure on $X$. Furthermore, if the parametric version of Problem \ref{pr:Runge} 
has an affirmative answer, then the inclusion \eqref{eq:Cont} is a weak homotopy equivalence.
\end{theorem}

Theorem \ref{th:ifthen} is proved in Sect.\ \ref{sec:proofs}.

%
%
We now consider more carefully the case when $X$ is a Stein manifold with $\dim X=3$. 
Let $L$ be a holomorphic line bundle on $X$ satisfying \eqref{eq:KX}, i.e., such that $K_X\otimes L^2$
is a trivial line bundle. (By the Oka-Grauert principle, every complex vector bundle on a Stein manifold 
carries a compatible structure of a holomorphic vector bundle; 
see \cite[Theorem 5.3.1]{Forstneric2017E}.) Note that $T^*X\otimes L$
admits a nowhere vanishing holomorphic section $\alpha$, i.e., 
an $L$-valued holomorphic $1$-form on $X$ 
(see \cite[Corollary 8.3.2]{Forstneric2017E}). Let $\xi=\ker\alpha\subset TX$.  
Then, $K_X\cong \Lambda^2\xi^* \otimes (TX/\xi)^* \cong  \Lambda^2\xi^* \otimes L^*$.
Since  $K_X\cong (L^*)^2$ by the assumption, we infer that $\Lambda^2\xi^* \otimes L$ is a 
trivial bundle. A trivialisation of $\Lambda^2\xi^* \otimes L$
is a $2$-form $\beta$ on $\xi$ with values in $L$ such that $\omega = \alpha\wedge\beta$
is a trivialisation of $K_X\otimes L^2$, i.e., $(\alpha,\beta)\in \Contf(X)$. 
Hence, the necessary condition \eqref{eq:KX} for the existence of an $L$-valued formal complex
contact structure on $X$ is also sufficient when $X$ is a Stein manifold and $\dim X=3$.

We denote by $\Contf(X,L)$ the subset of $\Contf(X)$ consisting of pairs of $L$-valued 
forms $(\alpha,\beta)\in \Contf(X)$. Clearly,  $\Contf(X,L)$ is a union of connected
components of $\Contf(X)$. We claim that the connected components
of $\Contf(X,L)$ coincide with the homotopy classes of trivialisations
of $K_X\otimes L^2$. One direction is obvious: given a homotopy
$(\alpha_t,\beta_t)\in \Contf(X,L)$ with $t\in [0,1]$, the family 
$\alpha_t\wedge \beta_t$ is a homotopy of  trivialisations of $K_X\otimes L^2$.
Conversely,  assume that $(\alpha_0,\beta_0), (\alpha_1,\beta_1) \in \Conth(X,L)$
and there is a homotopy $\omega_t$  of trivialisations of $K_X\otimes L^2$
with $\omega_0=\alpha_0\wedge \beta_0$ and $\omega_1=\alpha_1\wedge \beta_1$.
Since $\dim X=3$ and $X$ is Stein, it is homotopy equivalent to a $3$-dimensional CW complex.  
A simple topological argument in the line of 
\cite[proof of Corollary 8.3.2]{Forstneric2017E}
then shows that $\alpha_0$ and $\alpha_1$ can be connected
by a homotopy $\alpha_t$ of nowhere vanishing sections of 
$T^*X\otimes L$. Let $\xi_t=\ker \alpha_t\subset TX$ for $t\in [0,1]$.
Then, $\omega_t=\alpha_t\wedge \tilde \beta_t$ where $\tilde \beta_t$
is a trivialisation of $\Lambda^2\xi_t^* \otimes L$ and $\tilde \beta_0=\beta_0$.
At $t=1$ we have $\omega_1=\alpha_1\wedge \beta_1=\alpha_1\wedge \tilde\beta_1$,
and it follows that $\tilde \beta_1|_{\xi_1}=\beta_1|_{\xi_1}$. This proves the claim.

Recall that the isomorphism classes of complex (or holomorphic)
line bundles on a Stein manifold $X$ are in bijective correspondence with the elements
of $H^2(X;\Z)$ by Oka's theorem (see \cite[Theorem 5.2.2]{Forstneric2017E}).
The above observations yield the following homotopy classification of 
formal complex contact structures on Stein threefolds. 

%
%
\begin{proposition}\label{prop:basic}
If $X$ is a Stein manifold of dimension $3$, then the connected components of 
the space $\Contf(X)$ of formal complex contact structures on $X$
are in one-to-one correspondence with the following pairs of data:
\begin{enumerate}[\rm (i)]
\item an isomorphism class of a complex line bundle $L$ on $X$ satisfying 
$L^{2} \cong (K_X)^{-1}$, i.e., an element $c\in H^2(X;\Z)$ with $2c=c_1(TX)$, and
\item a choice of a homotopy class of trivialisations of the line bundle $K_X\otimes L^2$,
that is, an element of $[X,\C^*]=[X,S^1] = H^1(X;\Z)$.
\end{enumerate}
In particular, if $H^1(X;\Z)=0$ and $H^2(X;\Z)=0$ then the space $\Contf(X)$ is connected;
this holds for $X=\C^3$.
\end{proposition}

Theorem \ref{th:basic} and Proposition \ref{prop:basic} imply the following corollary.

%
%
\begin{corollary}\label{cor:basic}
Let $X$ be a Stein manifold of dimension $3$. 
Given a holomorphic line bundle $L$ on $X$ such that $(K_X)^{-1} \cong L^{2}$, there are a 
Stein domain $\Omega\subset X$ diffeotopic to $X$ and a holomorphic contact subbundle 
$\xi\subset T\Omega$ such that  $T\Omega/\xi$ is isomorphic to $L|_\Omega$.
Furthermore, given holomorphic $L$-valued contact forms $\alpha_0,\alpha_1$ on $X$
such that the map 
$\frac{\alpha_1\wedge d\alpha_1}{\alpha_0\wedge d\alpha_0}:X\to\C^*$ is null homotopic,
there are a Stein domain $\Omega\subset X$ as above and a homotopy $\alpha_t\in\Contf(X)$ 
$(t\in [0,1])$ connecting $\alpha_0|_\Omega$ to $\alpha_1|_\Omega$.
\end{corollary}

Since we must pass to Stein subdomains of $X$ when constructing
contact structures and homotopies between them, the following problem remains open.

\begin{problem}\label{prob:connected}
Let $X$ be a Stein manifold of dimension $3$ with $H^1(X;\Z)=H^2(X;\Z)=0$.
Is the space $\Conth(X)$ connected? In particular, is $\Conth(\C^3)$ connected?
\end{problem}

\begin{remark}\label{rem:Cor1.6}
Corollary \ref{cor:basic} gives a homotopy classification of
contact forms on Stein $3$-folds, but not necessarily of contact bundles.   
A holomorphic contact bundle $\xi$ on $X$ is determined by a holomorphic 
$1$-form $\alpha$ up to a nonvanishing factor $f\in \Oscr(X,\C^*)$.
Since $f\alpha \wedge d(f\alpha)=f^2\alpha\wedge d\alpha$, this changes the trivialisation 
of $K_X\otimes L^2$ by $f^2$. (More generally, if $\dim X=2n+1$ then the trivialisation
of $K_X\otimes L^{n+1}$ given by $\alpha\wedge (d\alpha)^n$ changes by the factor $f^{n+1}$.)
%
%
%
Hence, a homotopy class of holomorphic contact bundles on a Stein $3$-fold $X$ 
is uniquely determined by a pair $(c,d)$, where $c\in H^2(X;\Z)$ satisfies $2c=c_1(TX)$ 
and $d\in H^1(X;\Z)/2H^1(X;\Z)$. By Corollary \ref{cor:basic} every such pair is represented
by a holomorphic contact bundle on a Stein domain $\Omega\subset X$ diffeotopic to $X$.
\qed\end{remark}

We do not have a comparatively good classification result for $\Contf(X)$ 
on Stein manifolds of dimension five or more. Granted the necessary conditions
\eqref{eq:KX}, \eqref{eq:Lambda} for the normal bundle $L$, 
the existence and classification of complex symplectic $L$-valued 2-forms 
$\beta$ on the $2n$-plane bundle $\xi=\ker\alpha$ amounts to the analogous problem 
for sections of an associated fibre bundle with the fibre $GL_{2n}(\C)/Sp_{2n}(\C)$. 
We do not pursue this issue here. 

One may wonder to what extent it is possible to control the choice of the domain $\Omega\subset X$
in Theorem \ref{th:basic} and Corollary \ref{cor:basic}. In our proof, $\Omega$ arises as
a thin Stein neighbourhood of an embedded CW complex in $X$ which represents 
its Morse complex, so it carries all the topology of $X$. However, since a Mergelyan-type approximation
theorem is used in the construction, we do not know how large $\Omega$ can be.
We describe the construction more precisely at the end of this introduction and
supply references. 

Our method actually gives much more.
Assume that $X$ is an odd dimensional complex manifold (not necessarily Stein) 
and $W\subset X$ is a tamely embedded CW complex 
of dimension at most $\dim X$. (A suitable notion of tameness 
was introduced by Gompf \cite{Gompf2005,Gompf2013}.) 
Let $(\alpha,\beta)$ be a formal contact structure on $X$. 
After a small topological adjustment of $W$ in $X$, there is
a holomorphic contact form $\wt \alpha\in\Conth(\Omega)$ on a Stein thickening
$\Omega\subset X$ of $W$ such that $(\wt\alpha, d\wt\alpha)$ 
is homotopic to $(\alpha,\beta)$ in $\Contf(\Omega)$. 

This is illustrated most clearly by 
looking at holomorphic contact structures on neighbourhoods of totally real submanifolds.
 A real submanifold $M$ of class $\Cscr^1$ in a complex manifold $X$ is said to be {\em totally real}
if the tangent space $T_x M$ at any point $x\in M$ (a real vector subspace of $T_x X$)
does not contain any complex line. By Grauert \cite{Grauert1958AM}, such $M$ admits
a basis of tubular Stein neighbourhoods in $X$, the {\em Grauert tubes}. 
Every smooth $n$-manifold $M$ is a totally real submanifold of a 
Stein $n$-manifold: take the compatible real analytic structure on $M$, let
$M^\C$ be its complexification, and choose $X$ to be a Grauert tube around $M$ in $M^\C$.
The following is the $1$-parametric h-principle for germs of complex contact structures along 
a totally real submanifold; see Theorem \ref{th:trparametric} for the parametric case.

%
%
\begin{theorem}\label{th:trbasic} 
Let $M$ be a totally real submanifold of class $\Cscr^2$ in a complex manifold $X$. Every formal 
complex contact structure $(\alpha_0,\beta_0)\in \Contf(X)$ is homotopic  in $\Contf(X)$ to a holomorphic 
contact form $\alpha$ on a tubular Stein neighbourhood of $M$ in $X$. 
Furthermore, any two holomorphic contact forms $\alpha_0,\alpha_1$ on a neighbourhood of $M$ 
which are formally homotopic along $M$ are homotopic through a family of holomorphic contact forms 
$\alpha_t\in\Conth(\Omega)$ $(t\in[0,1])$ on a Stein neighbourhood $\Omega\subset X$ of $M$.
\end{theorem}

In dimension $3$ we have the following simpler statement in view of Proposition \ref{prop:basic}.

%
%
\begin{corollary}\label{cor:tubes3}
Let $X$ be a $3$-dimensional complex manifold and $M\subset X$ be a 
totally real submanifold of class $\Cscr^2$. Then, germs of complex contact 
forms on $X$ along $M$ are classified up to homotopy by pairs consisting
of a complex line bundle $L$ over a neighbourhood of $M$ 
satisfying $L^2|_M \cong (K_X)^{-1}|_M$ and an element of $H^1(M;\Z)$.
\end{corollary}

If $M$ is a totally real submanifold of maximal dimension $n$ in a complex $n$-manifold $X$, 
we have that $TX|_M\cong TM\oplus TM$ (since the complex structure operator $J$ on $TX$
induces an isomorphism of the tangent bundle $TM$ onto the normal bundle of $M$ in $X$). 
Replacing $X$ by a Grauert tube around $M$, it follows that 
$c_1(TX) = c_1(TX|_M)=c_1(TM\otimes \C)$,
so the canonical class of $X$ only depends on $M$. We shall see in Example \ref{ex:2sphere} 
that this is not the case in general for totally real submanifolds of lower dimension.

%
%
\begin{example}\label{ex:3sphere}
Let $X$ be a Grauert tube around  the $3$-sphere $S^3$. Then, 
$H^1(X;\Z)=H^1(S^3;\Z)=0$ and $H^2(X;\Z)=H^2(S^3;\Z)=0$. 
By Corollary \ref{cor:tubes3} there is a unique homotopy class of germs of complex 
contact structures around $S^3$ in $X$.
We get it for instance by taking a totally real embedding of $S^3$ into $\C^3$
(see \cite[Theorem 1.4]{Forstneric1986EM} or \cite[p.\ 193]{Gromov1986E})
and using the standard complex contact form $dz+xdy$ on $\C^3$.

It was shown by Eliashberg \cite{Eliashberg1989IM}
that there exist countably many homotopy classes of smooth contact structures on $S^3$.
By choosing them real analytic, we can complexity them to
obtain holomorphic contact structures on neighbourhoods of $S^3$ in $X$.
By what has been said above, these structures are homotopic to each 
other as holomorphic contact bundles. 
\qed\end{example}

%
%
\begin{example}\label{ex:2sphere}
Let $Y$ be a Grauert tube around the $2$-sphere $S^2$.
An explicit example is the complexified $2$-sphere 
\[
	Y=\{(z_0,z_1,z_2)\in\C^3: z_0^2+z_1^2+z_2^2=1\}.
\]
Recall that the holomorphic tangent bundle any smooth complex hypersurface in $\C^n$
is holomorphically trivial (see  \cite[Proposition 8.5.3, p.\ 370]{Forstneric2017E});
in particular, $TY$ is trivial. Let $\pi : X\to Y$ be a holomorphic line bundle; 
the isomorphism classes of such bundles correspond to the elements of $H^2(Y;\Z)=H^2(S^2;\Z)=\Z$. 
Considering $Y$ as the zero section of $X$, we can view 
$X$ as the normal bundle $N_{Y,X}$ of $Y$ in $X$. Since $TY$ is trivial,
the adjunction formula for the canonical bundle gives
\[
	K_X|_Y \cong K_Y \otimes (N_{Y,X})^{-1} = X^{-1}.
\] 
For each choice of the bundle $X\to Y$ with even Chern number $c_1(X)\in H^2(Y; \Z)=\Z$,  
$(K_X)^{-1}$ has a unique holomorphic square root $L$ with $c_1(L)=\frac{1}{2}c_1(X)$. 
By Corollary \ref{cor:tubes3} there is a holomorphic $L$-valued contact 
form on a neighbourhood of $S^2$ in $X$. A Stein tube around $S^2$ 
in the trivial bundle $X=Y\times \C$ can be represented as a domain in $\C^3$,
for example, as a tube around the standard $2$-sphere $S^2\subset \R^3\subset \C^3$.
The examples with nonzero  Chern classes clearly cannot be represented 
as domains in $\C^3$.
\qed\end{example}

%
%
\begin{example}\label{ex:circle}
Let $X$ be a $3$-dimensional Grauert tube around an embedded circle $S^1\subset X$.
In this case $H^2(X;\Z)=H^2(S^1;\Z)=0$,  and by Corollary \ref{cor:tubes3} the homotopy 
classes of holomorphic contact forms on $X$ along $S^1$ are classified by  
$H^1(X;\Z)=H^1(S^1;\Z)=\Z$. We can see them explicitly on $X=\C^* \times \C^2$ as follows.
Let $(x,y,z)$ be complex coordinates on $\C^3$. Set $S^1=\{(x,0,0)\in\C^3 :|x|=1\}$. 
For each $k\in\Z$ let
\[
 	\alpha_k =  
	\begin{cases} 
   		dz+\frac{1}{k+1} x^{k+1} dy  & \text{if } k \ne -1, \\
  		\frac{1}{\sqrt 2} \left(\frac{1}{x} dz + xdy\right)        & \text{if } k=-1.
  	\end{cases}
\]
Then, $\alpha_k\wedge d\alpha_k = x^k dx\wedge dy\wedge dz$
for every $k\in\Z$, so the homotopy class of the corresponding framing 
of the trivial bundle $X\times \C\to X$  equals $k$. 
By Remark \ref{rem:Cor1.6} the contact bundle $\xi_k=\ker\alpha_k$
on $\C^* \times \C^2$ is homotopic to $\xi_0$ if $k$ is even, and to $\xi_1\cong \xi_{-1}$ is
$k$ is odd. The bundles $\xi_0$ and $\xi_1$ are not homotopic to each other through 
contact bundles.

Note that the 1-form $\alpha_k$ for $k\ne -1$ is the pullback of the standard contact form
$\alpha_0=dz+xdy$ on $\C^3$ by the covering map $\C^* \times \C^2 \to \C^* \times \C^2$,  
$(x,y,z)\mapsto (x^{k+1}/(k+1),y,z)$. In order to understand 
$\alpha_{-1}$, consider the contact form on $\C^3$ given by
\[
	\beta=\cos x \,\cdotp dz+\sin x \,\cdotp dy.
\] 
It defines the standard structure on $\C^3$, because it is the pullback of 
$dz - ydx$ by the automorphism $(x,y,z) \to (x,y \cos x - z \sin x, y\sin x  + z \cos x)$.
%
%
%
%
Let $F:\C^3\to \C^*\times\C^2$ denote the universal covering map 
$F(x,y,z)=(\E^{\imath x},y,z)$. A calculation shows that $\beta=F^*\alpha'$, 
where $\alpha'$ is the contact form on $\C^*\times\C^2$ given by
\[
	\alpha'=\frac{1}{2}\left(x+\frac{1}{x}\right) dz + \frac{1}{2\imath}\left(x-\frac{1}{x}\right) dy,
	\qquad \alpha'\wedge d\alpha' = \frac{1}{\imath x} dx \wedge dy \wedge dz.
\]
Then, $\alpha_{-1}$ is homotopic to $\alpha'$ through the family
of contact forms on $\C^* \times \C^2$ defined by
\[
	\sigma_t = \frac{1}{\sqrt{2(1+t^2)}} 
	\left( \left(tx+\frac{1}{x}\right) dz +  \left(x-\frac{t}{x}\right) \E^{-\imath \pi t/2} dy \right),
	\quad t\in[0,1].
\]
We have $\sigma_0=\alpha_{-1}$, $\sigma_1=\alpha'$, and 
$\sigma_t\wedge d\sigma_t = \E^{-\imath \pi t/2}x^{-1} dx\wedge dy \wedge dz$ for all $t\in[0,1]$.
\qed\end{example}

%
%
\begin{example}\label{ex:3torus}
The previous example can be generalised to $(\C^*)^2\times \C$ and $(\C^*)^3$
which are complexifications of the $2$-torus and the $3$-torus, respectively.
Let us consider the latter.
Denote by $T^k$ the $k$-dimensional torus, the product of $k$ copies of the circle $S^1$. 
The domain $X=(\C^*)^3$ is a Stein tube around the standard totally real embedding 
$T^3\hra \C^3$ onto the distinguished boundary of the polydisc.
We have $H^2(X;\Z)=H^2(T^3;\Z)=\Z^3$ and $H^1(X;\Z)= H^1(T^3;\Z)= \Z^3$
(see Rotman \cite[p.\ 404]{Rotman1988}). Clearly, $K_X$ is trivial,
and since $H^2(X;\Z)$ is a free abelian group, its only square root is the
trivial bundle. Hence by \eqref{eq:KX} all contact forms on $X$ assume values in the trivial bundle, 
and we have $\Z^3$-many homotopy classes of trivialisations of the latter. 
Consider the following family of contact forms on $X=(\C^*)^3$, 
where $(k,l,m)\in\Z^3$:
\[
 	\alpha_{k,l,m} =  
	\begin{cases} 
   		z^m dz+\frac{1}{k+1} x^{k+1} y^l dy  & \text{if } k \ne -1, \\
  		\frac{1}{2x} z^m dz + x y^l dy              & \text{if } k=-1,
  	\end{cases}
\]
A calculation shows that $\alpha_{k,l,m} \wedge d\alpha_{k,l,m}=x^k y^l z^m dx\wedge dy\wedge dz$,
so this family provides all possible homotopy classes of framings of the trivial bundle
$X\times \C$. 
\qed\end{example}

The above examples suggest that in many natural cases one can find globally
defined holomorphic contact forms representing all homotopy
classes in Proposition \ref{prop:basic}. 

\begin{problem} 
Is it possible to represent every homotopy class of formal complex contact structures 
on an affine algebraic manifold by an algebraic contact form?
\end{problem}

Our proofs of Theorems \ref{th:trbasic} and \ref{th:trparametric} 
proceed by triangulating the manifold $M$ and inductively 
deforming a formal contact structure $(\alpha,\beta)$ to an almost contact structure along $M$
(see Definition \ref{def:almostcontact} (b) for this notion).
We show that the open partial differential relation of first order, controlling 
the almost contact condition on a totally real disc, is ample in the coordinate directions; 
see Lemma \ref{lem:ample}. Hence, Gromov's h-principle \cite{Gromov1986E,Gromov1973IZV}
can be applied to extend an almost contact structure from the boundary
of a cell to the interior, provided that it extends as a formal complex contact structure; 
see Lemma \ref{lem:main}. Finally, approximating an almost contact form $\alpha$ on $M$
sufficiently closely in the fine $\Cscr^1$ topology by a holomorphic $1$-form $\wt \alpha$ 
ensures that $\wt \alpha$ is a contact form on a neighbourhood of $M$ in $X$.
The same arguments apply to families of such forms, thereby yielding the
parametric h-principle in Theorem \ref{th:trparametric}.

A similar method is used to prove Theorems \ref{th:basic} and \ref{th:parametric}
(see Sect.\ \ref{sec:proofs}). The inductive step amounts to extending 
a holomorphic contact form $\alpha$ from a neighbourhood of a compact strongly 
pseudoconvex domain $W$ in $X$ across a handle whose core is a totally real disc $M$ 
attached with its boundary sphere $bM$ to $bW$. 
More precisely, $M\setminus bM\subset X\setminus W$, the attachment is $J$-orthogonal 
along $bM$ (where $J$ denotes the almost complex structure on $X$),  
and $bM$ is a Legendrian submanifold of the strongly 
pseudoconvex hypersurface $bW$ with its smooth contact structure
given by the complex tangent planes.
The union $W\cup M$ then admits a basis of tubular Stein neighbourhoods
(see \cite{Eliashberg1990,ForstnericKozak2003}).
Assuming that $\alpha$ extends to $M$ as a formal contact structure, 
Lemma \ref{lem:hprinciple-AC} furnishes an almost contact extension. 
Finally, by Mergelyan's theorem we can approximate $\alpha$ in the $\Cscr^1$ topology
on $W\cup M$ by a holomorphic contact form $\wt \alpha$ on a 
Stein neighbourhood of $W\cup M$.

With these analytic tools in hand, Theorems \ref{th:basic} and \ref{th:parametric}
are proved by following the scheme developed by Eliashberg \cite{Eliashberg1990}
in his landmark construction of Stein manifold structures on any 
smooth almost complex manifold $(X,J)$ with the correct handlebody structure.
(The special case $\dim X=2$ is rather different and was explained by Gompf
\cite{Gompf1998,Gompf2005,Gompf2013}, but this is not relevant here.)
A more precise explanation of Eliashberg's construction was given by Slapar and the author 
\cite{ForstnericSlapar2007MRL,ForstnericSlapar2007MZ} in their proof of the 
{\em soft Oka principle} for maps from any Stein manifold $X$ to an arbitrary complex manifold $Y$. 
Expositions are also available in the monographs by Cieliebak and Eliashberg
\cite[Chap.\ 8]{CieliebakEliashberg2012} and the author \cite[Secs.\ 10.9--10.11]{Forstneric2017E}.

Finally, the proof of Theorem \ref{th:ifthen} (see Sect.\ \ref{sec:proofs})
follows the induction scheme used in Oka theory; see \cite[Sect.\ 5]{Forstneric2017E}. 
Besides the tools already mentioned above,
an additional ingredient is a new gluing lemma for holomorphic contact forms;
see Lemma \ref{lem:gluingcontact}.

%
%
\section{Germs of complex contact structures on domains in $\R^{2n+1}\subset \C^{2n+1}$} 
\label{sec:tr}

We denote the complex variables on $\C^{n}$ by $z=(z_1,\ldots,z_{n})$ with $z_i = x_i+\imath y_i$
for $i=1,\ldots,n$, where $\imath=\sqrt{-1}$. We shall consider $\R^{n}$ as the standard real subspace 
of $\C^{n}$. 

Let $D$ be a compact set in $\R^{2n+1}$ $(n\in\N)$ which is 
the closure of a domain with piecewise $\Cscr^1$ boundary. 
We shall denote by $bD$ the boundary of $D$. In this section 
we consider the problem of approximating a holomorphic contact form $\alpha$, defined
on a neighbourhood of a compact subset $\Gamma\subset bD$, 
by a holomorphic contact form $\wt \alpha$ defined on a neighbourhood of $D$ 
in $\C^{2n+1}$, provided that $\alpha$ admits a formal 
contact extension to $D$ in the sense of Definition \ref{def:formal}.
(For applications in this paper, it suffices to consider the case when $D$ is the standard 
handle $D^m\times D^{d} \subset \R^{2n+1}$ of some index $m\in \{1,\ldots,2n+1\}$ 
and $d=2n+1-m$, where $D^m\subset \R^m$ and $D^d\subset \R^d$ are closed unit balls 
in the respective spaces, and $\Gamma=bD^m \times D^d$ is the attaching set of the handle.)
We will show that the parametric h-principle holds in this problem (see Lemma \ref{lem:main}).

We begin with preliminaries.
Let $l\in \N$, and let $K$ be a closed set in a complex manifold $X$. 
A function $f$ of class $\Cscr^{l}$ on an open neighbourhood $U\subset X$ of $K$
is said to be {\em $\dibar$-flat to order $l$} on $K$ if the jet of $\dibar f$
of order $l-1$ vanishes at each point of $K$. In any system of local
holomorphic coordinates $z=(z_1,\ldots, z_n):V\to\C^n$ on $X$ centred at a point $x_0\in K$, 
this means that the value and all partial derivatives of order up to $l-1$ 
of the functions $\di f/\di \bar z_j= \frac{1}{2}\left(\di f_{x_j}+\imath\, \di f_{y_j}\right)$ 
$(j=1,\ldots,n)$ vanish at each point $x\in K\cap V$. In particular, such $f$ satisfies the 
Cauchy-Riemann equations at every point  $x\in K\cap V$:
\[
	\frac{\di f}{\di z_j}(x)=\frac{\di f}{\di x_j}(x)=\frac{1}{\imath} \frac{\di f}{\di y_j}(x),
	\qquad j=1,\ldots,n.
\] 
If $f$ is smooth of class $\Cscr^\infty$ and the above holds for all $l\in\N$,  
then $f$ is said to be {\em $\dibar$-flat} (to infinite order) on $K$.

Assume now that $D\subset \R^{2n+1}\subset \C^{2n+1}$ is a compact domain with piecewise
$\Cscr^1$ boundary in $\R^{2n+1}$. It is classical (see e.g.\ 
\cite[Lemma 4.3]{HormanderWermer1968} or \cite[Proposition 5.55]{CieliebakEliashberg2012}) 
that every function $f:D\to\C$ of class $\Cscr^l$ extends
to a $\Cscr^l$ function $F:\C^{2n+1}\to \C$ which is $\dibar$-flat to order $l$ on $D$.
When $f$ is of class $\Cscr^\infty$, we can obtain such an extension explicitly by first extending
$f$ to a smooth function on $\R^{2n+1}$ and setting
\[
	F(x+\imath y) = \sum_{|I|\le l} \frac{1}{I!} \frac{\di^{|I|}f}{\di x^I}(x) \, \imath^{|I|} y^I
	= f(x) + \imath \sum_{i=1}^{2n+1} \frac{\di f}{\di x_i}(x) y_i + O(|y|^2). 
\]
Here, $I=(i_1,\ldots,i_{2n+1})\in \Z_+^{2n+1}$, $|I|=i_1+\cdots+ i_{2n+1}$, 
$\frac{\di^{|I|}f}{\di x^I}(x)=\frac{\di^{|I|}f}{\di x_1^{i_1}\cdots \, \di x_n^{i_{2n+1}}}$,
and $y^I =y_1^{i_1}\cdots y_n^{i_{2n+1}}$. 
If $f$ is only of class $\Cscr^l$ then a $\dibar$-flat extension is 
obtained by applying Whitney's jet-extension theorem \cite{Whitney1934TAMS}
to the jet on the right hand side above.

A smooth differential $(1,0)$-form 
\begin{equation}\label{eq:alpha2}
	\alpha=\sum_{i=1}^{2n+1} a_i(z) dz_i
\end{equation}
on a neighbourhood of $D$ in $\C^{2n+1}$ is said to be {\em $\dibar$-flat to order $l$}
on $D$ is every coefficient function $a_i$ is such. 
Every smooth $(1,0)$-form defined on $D\subset \R^{2n+1}$ extends to
a $\dibar$-flat $(1,0)$-form on $\C^{2n+1}$ by taking $\dibar$-flat extensions of its coefficient.
Assume that $\alpha$ is such. In view of the Cauchy-Riemann equations we have for each $x\in D$ that
\begin{equation}\label{eq:dalpha}
	d\alpha(x) = \di\alpha(x) = 
	\sum_{1\le i<j\le 2n+1} 
	\left( \frac{\di a_j}{\di x_i}(x) - \frac{\di a_i}{\di x_j}(x) \right) dz_i\wedge dz_j.
\end{equation}
Write $p_{i,j}(x)=\frac{\di a_i}{\di x_j}(x)$ and set
\begin{equation}\label{eq:betaij}
	\beta_{i,j}(x) := p_{j,i}(x) - p_{i,j}(x) = \frac{\di a_j}{\di x_i}(x) - \frac{\di a_i}{\di x_j}(x).
\end{equation}	
With this notation, we have for all $x\in D$ that
\begin{equation}\label{eq:dalpha2}
	d\alpha(x) = \beta(x) = \sum_{1\le i<j\le 2n+1} \beta_{i,j}(x) \, dz_i\wedge dz_j
\end{equation}
and 
\begin{equation}\label{eq:betaton}
	(d\alpha)^n(x) = \beta^n(x) = \sum_{i=1}^{2n+1}  b_i(x) \, 
	dz_1\wedge\cdots \wh{dz_i}\cdots \wedge dz_{2n+1},
\end{equation}
where $\wh {dz_i}$ indicates that this term is omitted. Every coefficient $b_i(x)$  
in \eqref{eq:betaton} is a homogeneous polynomial of order $n$ in the coefficients $\beta_{j,k}$
of $\beta =d\alpha$ \eqref{eq:dalpha}, obtained as follows. Let $\Pcal=\{A_1,\ldots, A_n\}$ 
be a partition of the set $\{1,2\ldots,2n+1\}\setminus \{i\}$ into a union of $n$ 
pairs $A_k = (i_k,j_k)$ $(k=1,\ldots, n)$, with $i_k<j_k$. Then, 
\begin{equation}\label{eq:bi}
	b_i(x) = n!\,  \sum_\Pcal \prod_{(i_k,j_k) \in \Pcal} \beta_{i_k,j_k}(x)
	= n! \, \sum_\Pcal \prod_{(i_k,j_k) \in \Pcal} \left(p_{j_k,i_k}(x) - p_{i_k,j_k}(x) \right)
\end{equation}
for all $x\in D$. Finally, from \eqref{eq:dalpha} and  \eqref{eq:betaton} we obtain for all $x\in D$ that
\begin{equation}\label{eq:alphadalpha}
	\alpha(x) \wedge (d\alpha)^n(x) = \alpha(x) \wedge \beta^n (x) =
	\biggl(\, \sum_{i=1}^{2n+1} (-1)^{i-1} a_i(x)b_i(x) \!\biggr)
	dz_1\wedge \cdots\wedge dz_{2n+1}.
\end{equation} 

A smooth $(1,0)$-form $\alpha$ on $\C^{2n+1}$, defined on a neighbourhood of $D\subset \R^{2n+1}$ 
and $\dibar$-flat on $D$ to the first order, is said to be an {\em almost contact form} on $D$ if 
\begin{equation}\label{eq:almostcontact}	
	\alpha \wedge (d\alpha)^n \ne 0 \ \ \text{at every point of $D$}.
\end{equation}
Note that $d\alpha|_D=\di\alpha|_D$.
Approximating $\alpha$ sufficiently closely in the $\Cscr^1$ topology on $D$ 
by a holomorphic $1$-form $\wt\alpha$ gives a holomorphic contact structure 
$\tilde\xi=\ker\wt \alpha$ on a neighbourhood of $D$ in $\C^{2n+1}$. 
If the coefficients of $\alpha$ are real analytic, then the complexification
of $\alpha$ defines a holomorphic contact structure near $D$. 

We see from \eqref{eq:betaij}, \eqref{eq:bi}, and \eqref{eq:alphadalpha} that
the condition \eqref{eq:almostcontact} depends only on the first order jet of the restrictions $a_i|_{D}$ 
of the coefficients of $\alpha$ to $D$, so  it defines an open set  in the space of 
$1$-jets of $1$-forms on $D$. More precisely, we may view $\alpha|_D$ 
as a smooth section $x\mapsto (x,a_1(x),\ldots, a_{2n+1}(x))$
of the trivial bundle $E=D\times \C^{2n+1}\to D$.
Let $E^{(1)}\to E$ be the bundle of $1$-jets of sections of $E\to D$.
The fibre of $E^{(1)}$ over a point $(x,a)\in E=D\times \C^{2n+1}$ (with $a=(a_1,\ldots, a_{2n+1}))$ 
consists of all matrices $p=(p_{i,j})\in \C^{(2n+1)\cdotp(2n+1)}$.
A section $D\to E^{(1)}$ is a map $x\mapsto (x,a(x),p(x))\in E^{(1)}$,
where $a:D\to\C^{2n+1}$ and $p:D\to  \C^{(2n+1)\cdotp(2n+1)}$. 
Such a section is said to be {\em holonomic} if $p(x)$ is the $1$-jet of $a(x)$
for each $x\in D$, that is, $p_{i,j}(x)=\frac{\di a_i}{\di x_j}(x)$ for all $i,j=1,\ldots, 2n+1$.
Let $\Rcal$ be the open subset of $E^{(1)}$ defined by
\begin{equation}\label{eq:Ccal}
	\Rcal = \biggl\{(x,a,p) \in E^{(1)} :  \sum_{i=1}^{2n+1} (-1)^{i-1} a_i b_i \ne 0\biggr\},
\end{equation}
where each $b_i$ is determined by $p=(p_{j,k})$ according to the formula \eqref{eq:bi}
(ignoring the base point $x$). Thus, $\Rcal$ is an  open differential relation of first order
in $E^{(1)}$ which controls the contact condition for $\dibar$-flat 1-forms along $D$.

%
%
\begin{lemma}\label{lem:ample}
The partial differential relation $\Rcal$ defined by \eqref{eq:Ccal} 
is ample in the coordinate directions (in the sense of M.\ Gromov \cite{Gromov1973IZV,Gromov1986E}).
\end{lemma}

\begin{proof} 
Choose an index $i\in \{1,\ldots, 2n+1\}$. Write $p=(p_1,\ldots,p_{2n+1})$ and 
$p_j=(p_{j,1},\ldots,  p_{j,2n+1}) \in\C^{2n+1}$ for $j=1,\ldots,2n+1$.
Consider a restricted $1$-jet of the form $e=(x,a,p_1,\ldots, \wh{p_i},\ldots,p_{2n+1})$
where the vector $p_i$ is omitted. Set
\begin{equation}\label{eq:restricted}
	\Rcal_e = \{p_i \in \C^{2n+1} : (x,a,p_1,\ldots, p_{i-1},p_i,p_{i+1},\ldots,p_{2n+1}) \in \Rcal\}. 
\end{equation}
The differential relation $\Rcal$ is said to be {\em ample in the coordinate directions} if every set 
$\Rcal_e$ of this type  is either empty, or else the convex hull of 
each of its connected components equals $\C^{2n+1}$.
In the case at hand, we see from \eqref{eq:bi} and \eqref{eq:alphadalpha} 
that the function 
\[	
	h(a,p)=\sum_{j=1}^{2n+1} (-1)^{i-1} a_j b_j(p),
\]
where $b_j=b_j(p)$ is determined by \eqref{eq:bi}, is affine linear in 
$p_i = (p_{i,1},\ldots,  p_{i,2n+1})$. Indeed, every $p_{i,j}$ appears at most once 
in each of the products in \eqref{eq:bi}. Since
\[
	\Rcal_e = \{p_i \in \C^{2m+1} : h(a,p_1,\ldots, p_i,\ldots,p_{2n+1})\ne 0\},
\]
it follows that $\Rcal_e$ is either empty or else the complement of a complex affine
hyperplane in $\C^{2n+1}$; in the latter case its convex hull equals $\C^{2n+1}$.
This proves Lemma \ref{lem:ample}.
\end{proof}

\begin{remark}\label{rem:false}
Note that the real analogue of Lemma \ref{lem:ample} is false. For this reason,
the corresponding h-principle for real contact structures, due to Gromov 
\cite{Gromov1969}, does not hold for compact smooth manifolds, but only on
open ones. 
\end{remark}

In order to apply this lemma, we need the following observation. Let $\alpha$ be a $1$-form
\eqref{eq:alpha2} with smooth coefficients $a=(a_1,\ldots, a_{2n+1}): D\to\C^{2n+1}$, and let
\begin{equation}\label{eq:beta2}
	\beta(x) = \sum_{1\le i<j\le 2n+1} \beta_{i,j}(x) \, dz_i\wedge dz_j,\qquad x\in D, 
\end{equation}
be a smooth $2$-form on $D$. (At this point we consider forms with values in the trivial line bundle.)
Note that the linear projection 
$\C^{(2n+1)^2} \ni (p_{i,j})\mapsto (\beta_{i,j}=p_{j,i}-p_{i,j})\in \C^{n(2n+1)}$ is surjective
and hence a Serre fibration, i.e., it enjoys the homotopy lifting property.
In particular, we may write $\beta_{i,j}=p_{j,i}-p_{i,j}$ for some smooth functions $p_{i,j}$ on $D$. 
Let $p=(p_{i,j}):D\to \C^{(2n+1)^2}$.
It then follows from the definition of the differential relation $\Rcal$ (see \eqref{eq:Ccal}) that 
$(\alpha,\beta)$ is a formal contact structure on $D$ (see Definition \ref{def:formal}), i.e.,  
\begin{equation}\label{eq:alphabeta}
	\alpha \wedge \beta^n \ne 0 \quad \text{on $D$},
\end{equation}
if and only if the map $x\mapsto (x,a(x), p(x))$ is a (not necessarily holonomic) section of $\Rcal$. 
Note that the condition \eqref{eq:alphabeta} is purely algebraic and does not
depend on the particular choices of extensions of $\alpha$ and $\beta$ to a neighbourhood of $D$.

A seminal result of M.\ Gromov says that sections of an ample open differential relation $\Rcal$ 
of first order satisfy all forms of the h-principle (see \cite[Sect.\ 2.4]{Gromov1986E}, \cite[Sect.\ 18.2]{EliashbergMishachev2002}, or \cite[Theorem 4.2]{Spring2010}).
This means that every section of $\Rcal$ is homotopic through sections of $\Rcal$
to a holonomic section, the homotopy can be chosen fixed on a compact
subset of the base domain where the given section is already holonomic, and a similar
statement holds for families of sections, where the homotopy is kept fixed on the set of holonomic
sections. The basic technical result is the following; we state it for the case at hand.  
(See for instance \cite[Lemma 3.1.3, p.\ 339]{Gromov1973IZV} which is stated
for the special case when $D$ is a compact cube and $\Gamma=bD$;
the general case follows by induction on a suitable triangulation of the pair $(D,\Gamma)$.
A brief survey is also available in \cite[Sect.\ 1.10]{Forstneric2017E}.)

%
%
\begin{lemma}\label{lem:main}
Let $D\subset \R^{2n+1}$ be a compact domain with piecewise
$\Cscr^1$ boundary, and let $\Gamma\subset bD$ be the closure of an open 
subset of $bD$ with piecewise $\Cscr^1$ boundary. 
Assume that $\alpha$ is a smooth $\dibar$-flat $(1,0)$-form and 
$\beta$ is a smooth $(2,0)$-form on a neighbourhood of $D$ in $\C^{2n+1}$ 
(see \eqref{eq:alpha2}, \eqref{eq:beta2}) such that \eqref{eq:alphabeta} holds 
and $d\alpha(x)=\beta(x)$ for all $x\in \Gamma$, i.e.,
\[
	\beta_{i,j}(x) = \frac{\di a_j}{\di x_i}(x) - \frac{\di a_i}{\di x_j}(x) \quad
	\text{for all $x\in \Gamma$ and $i,j=1,\ldots,2n+1$}.
\]
Given $\epsilon>0$ there is a homotopy $(\alpha_t,\beta_t)$ $(t\in [0,1])$ 
of pairs of forms of the same type satisfying the following conditions. 
\begin{enumerate}[\rm (i)]
\item $(\alpha_0,\beta_0)=(\alpha,\beta)$.
\item $\alpha_t(x) \wedge \beta_t(x)^n \ne 0$ for all $x\in D$ and $t\in [0,1]$.
\item $|\alpha_t(x)-\alpha(x)|<\epsilon$ for all $x\in D$ and $t\in [0,1]$. 
\item The homotopy is fixed for $x\in \Gamma$.
\item $\beta_1=d\alpha_1$ holds at all points of $D$, i.e., $\alpha_1$ is an almost contact form on $D$.
\end{enumerate}	
Assume furthermore that $P$ is a compact Hausdorff space, $Q\subset P$ is a closed subspace, 
and $\{(\alpha_p,\beta_p)\}_{p\in P}$ is a continuous family of data as above 
such that for every $p\in Q$ we have that $d\alpha_p=\beta_p$ on $D$.
Then, there is a homotopy $(\alpha_{p,t},\beta_{p,t})$ $(t\in [0,1])$ which is fixed (independent of $t$) 
for every $p\in Q$ and satisfies conditions (i)--(v) for every $p\in P$.
\end{lemma}

In condition (iii) we use the Euclidean norm for the coefficient vector of the $1$-form $\alpha_t-\alpha$,
that is, $\alpha_t$ is uniformly $\epsilon$-close to $\alpha=\alpha_0$ on $D$ for all $t\in [0,1]$.
Note however that  in general $\alpha_1$ cannot be chosen $\Cscr^1$-close to $\alpha$.

Lemma \ref{lem:main} is proved by applying the h-principle on $\R^{2n+1}$ and
then extending the resulting forms $\dibar$-flatly to a neighbourhood in $\C^{2n+1}$.

%
%
\section{Asymptotically holomorphic and almost contact forms}\label{sec:AC}

We now introduce a general notion of an {\em almost contact form} along a closed
subset $M$ in a complex manifold $X$ (see Definition \ref{def:almostcontact}).
This is necessary since we shall be applying 
coordinate changes which are asymptotically holomorphic on $M$, 
but not necessarily holomorphic. For simplicity we discuss  scalar-valued forms, 
although the same notions apply to differential forms with values in any  
holomorphic line bundle on $X$. However,
Lemma \ref{lem:pullback} and Corollary \ref{cor:almostcontact}
only apply to scalar-valued forms and will be used locally.

A smooth differential $m$-form $\alpha$ on a complex manifold $X$ 
decomposes uniquely as the sum $\alpha=\sum_{p+q=m} \alpha^{p,q}$ of its 
$(p,q)$-homogeneous parts. In local holomorphic coordinates $z=(z_1,\ldots,z_n)$ 
on $X$ we have 
\[
	\alpha^{p,q}= \sum 
	a_{I,J} \, dz_{i_1}\wedge \cdots \wedge dz_{i_p}
	\wedge d\bar z_{j_1} \wedge \cdots \wedge d\bar z_{j_q}
\]
for some smooth coefficient functions $a_{I,J}$. In particular, for a $1$-form $\alpha$ we have
\begin{equation}\label{eq:1form}
	\alpha = 	\sum_{i=1}^n a_i dz_i + \sum_{i=1}^n b_id\bar z_i = \alpha^{1,0}+\alpha^{0,1}.
\end{equation}
The exterior derivative on $X$ splits as $d=\di+\dibar$. If $\alpha$ is a $1$-form then
\[ 
	(d\alpha)^{2,0} = \di\alpha^{1,0},\quad
	(d\alpha)^{1,1} = \di\alpha^{0,1}+\dibar\alpha^{1,0}, \quad
	(d\alpha)^{0,2} = \dibar  \alpha^{0,1}.
\]

%
%
\begin{definition}\label{def:AH}
Let $M$ be a closed subset of a complex manifold $X$.
\begin{enumerate}[\rm (a)]
\item 
A smooth $m$-form $\alpha$, defined on a neighbourhood of $M$ in $X$, is
{\em of type $(m,0)$ on $M$} if 
\[
	\alpha|_M = \alpha^{m,0}|_M.
\]
The space of such $m$-forms (on variable neighbourhoods of $M$) is denoted $\Ecal^{m,0}(M,X)$.
\item
A smooth $1$-form $\alpha$, defined on a neighbourhood of $M$ in $X$,
is {\em asymptotically holomorphic} (of order $1)$ on $M$ if
for every point $x_0\in M$ there is a holomorphic coordinate system on $X$ around $x_0$
in which $\alpha$ has the form \eqref{eq:1form} and the following conditions hold
for $i=1,\ldots, n$:
\begin{equation}\label{eq:vanishing}
	\dibar a_i(x_0)=0,\qquad b_i(x_0)=0,\qquad db_i(x_0)=0.
\end{equation}
The space of all such forms on variable neighbourhoods of $M$ is denoted $\AH^1(M,X)$.
\end{enumerate}
\end{definition}

The first two conditions in \eqref{eq:vanishing} are equivalent to $\alpha\in \Ecal^{1,0}(M,X)$
and $\dibar\alpha^{1,0}|_M=0$, so $d\alpha^{1,0}|_M=\di \alpha^{1,0}|_M$.
The last condition in  \eqref{eq:vanishing}  implies $d\alpha^{0,1}|_M=0$, but the converse
is not true since $\dibar\alpha^{0,1}|_M=0$ holds under the weaker condition 
$\frac{\di b_i}{\di \bar z_k}=\frac{\di b_k}{\di \bar z_i}$ on $M$ for all $i,k=1,\ldots,n$.
In particular, we have that
\[
	\AH^1(M,X) \subset \{\alpha \in \Ecal^{1,0}(M,X): d\alpha^{0,1}|_M=0,\ \
	\dibar \alpha^{1,0}|_M=0,\ \ 
	d\alpha\in\Ecal^{2,0}(M,X)\}.
\]

Assume now that $X$ and $Y$ are complex manifolds and $F:X\to Y$ is smooth map. 
Let $M$ be a closed subset of $X$. We say $F$ is {\em $\dibar$-flat} (or asymptotically holomorphic)
to order $k\in\N$ on $M$ if, in any pair of holomorphic coordinates on the two manifolds, 
we have that
\begin{equation}\label{eq:dibarflatext}
	D^{k-1}(\dibar F)|_M=0,
\end{equation}
where $D^{k-1}$ is the total derivative of order $k-1$ applied to the components
$\di F_i/\di \bar z_j$ of $\dibar F$. The chain rule shows that this notion is independent of the 
choice of coordinates.


The following lemma shows in particular that condition \eqref{eq:vanishing} 
defining the class $\AH^1(M,X)$ is invariant under $\dibar$-flat coordinate changes.

%
%
\begin{lemma}\label{lem:pullback}
Assume that $X$ and $Y$ are complex manifolds and $F:X\to Y$ is a $\Cscr^2$ map
which is $\dibar$-flat to order $2$ on a closed subset $M\subset X$. 
Set $M'=\overline{F(M)}\subset Y$. If $\alpha\in \AH^1(M',Y)$ then 
$F^*\alpha\in  \AH^1(M,X)$ and 
\[
	d (F^*\alpha)|_M= \di ((F^*\alpha)^{1,0}) |_M = F^*(\di \alpha^{1,0}|_{M'}).
\]
\end{lemma}

\begin{proof}
Fix a point $x_0\in M\subset X$ and let $y_0=F(x_0)\in M'\subset Y$.
By the assumption there are holomorphic coordinates $w=(w_1,\ldots,w_n)$ on a
neighborhood $U$ of $y_0$ in $Y$ such that
\[
	\alpha = 	\sum_{i=1}^n a_i\, dw_i + \sum_{i=1}^n b_i\, d\overline w_i = \alpha^{1,0}+\alpha^{0,1},
\]
where the coefficients satisfy the following conditions (see \eqref{eq:vanishing}): 
\[
	\dibar a_i(y_0)=0,\qquad b_i(y_0)=0,\qquad db_i(y_0)=0.
\]
The pullback form $\wt\alpha=F^*\alpha$ on $F^{-1}(U)\subset X$ equals
\[
\begin{aligned}
	\wt\alpha &= \sum_{i=1}^n \left[ (a_i\circ F)\, dF_i + (b_i\circ F)\, d\overline {F}_i\right] \\
	&= \sum_{i=1}^n \left[ (a_i\circ F)\, \di F_i + (b_i\circ F) \,\di \overline F_i\right] 
	+  \sum_{i=1}^n \left[ (a_i\circ F)\, \dibar F_i + (b_i\circ F) \, \dibar \, \overline F_i\right] \\
	& =  \wt\alpha^{1,0} + \wt\alpha^{\, 0,1}.
\end{aligned}
\]
At the point $x_0\in M$ we have that $b_i\circ F(x_0)=0$ and $\dibar F_i(x_0)=0$  for all $i$, 
and hence
\[
	\wt\alpha^{1,0}(x_0) = \sum_{i=1}^n  a_i(y_0) \di F_i (x_0) = F^*(\alpha^{1,0})(x_0), 
	\qquad \wt\alpha^{0,1}(x_0)=0.
\]
Furthermore, since $db_i(y_0)=0$ and $d(\dibar F_i)(x_0)=0$ for all $i$, a simple calculation
shows that the coefficients of $\wt\alpha^{0,1}$ in any holomorphic coordinate system
on $X$ around $x_0$ vanish to the second order at $x_0$.
Finally, consider the $(1,1)$-form
\[
	\dibar\, \wt\alpha^{1,0} = 
	\sum_{i=1}^n \left( 
		\, \dibar(a_i\circ F)\wedge \di F_i
		+ (a_i\circ F) \dibar\di F_i
        		+ \dibar (b_i\circ F) \wedge  \di \overline F_i
		+ (b_i\circ F) \dibar\di \overline F_i 
		\right).
\]
We have that
\[
	\dibar(a_i\circ F)(x_0) = 
		\sum_{k=1}^m \left( \frac{\di a_i}{\di w_k}(y_0) \dibar F_k(x_0) + 
				\frac{\di a_i}{\di \overline w_k}(y_0) \dibar (\overline F_k)(x_0) \right) 
				=0,
\]
so the first term in the above sum for $\dibar\, \wt\alpha^{1,0}$ vanishes at $x_0$.
The other terms vanish as well since $F$ is $\dibar$-flat to the second order at $x_0$.
This shows that $\wt\alpha=F^*\alpha$ is asymptotically holomorphic at $x_0$. 
Since the point $x_0\in M$ was arbitrary, this completes the proof.
\end{proof}

%
%
\begin{definition}\label{def:almostcontact}
Let $X^{2n+1}$ be a complex manifold and $M$ be a closed subset of $X$.
\begin{enumerate}[\rm (a)]
\item 
A pair $(\alpha,\beta)$ with $\alpha\in \Ecal^{1,0}(M,X)$ and $\beta\in \Ecal^{2,0}(M,X)$
(see Definition \ref{def:AH}) is a {\em formal complex contact structure on $M$} if 
\begin{equation}\label{eq:formalcontact}
	\alpha\wedge \beta^n = 
	\alpha^{1,0}\wedge (\beta^{2,0})^n \ne 0  \quad \text{holds at every point of $M$}.
\end{equation}
We denote by $\Contf(M,X)$ the space of formal contact structures on $M\subset X$.
\vspace{1mm}
\item
An asymptotically holomorphic $1$-form $\alpha\in \AH^1(M,X)$ 
(see Definition \ref{def:AH} (b)) is an {\em almost contact form on $M$} if 
\begin{equation}\label{eq:almostcontact2}
	\alpha \wedge (d\alpha)^n  \ne 0  \ \ \text{holds at every point of $M$}.
\end{equation}
We denote the space of almost contact forms on $M$ by $\AC(M,X)$.
\end{enumerate}
\end{definition}

\begin{remark}\label{rem:observe}
Note that for every $(\alpha,\beta)\in \Contf(M,X)$ the pair $(\alpha^{1,0},\beta^{2,0})$
is a formal contact structure on an open neighbourhood of $M$ in $X$ 
(since $\alpha^{1,0}\wedge (\beta^{2,0})^n \ne 0$ is an open condition). 
Likewise, $\AC(M,X)$ is an open subset of  $\AH^1(M,X)$ in the fine $\Cscr^1$ topology
on $M$. For $\alpha\in \AH^1(M)$, the almost contact condition \eqref{eq:almostcontact2} 
is equivalent to
\[
	\alpha^{1,0} \wedge (d\alpha^{1,0})^n =  \alpha^{1,0} \wedge (\di \alpha^{1,0})^n 
	\ne 0  \quad \text{on $M$}.
\]
Hence, this notion generalises the one introduced in  
Sect.\ \ref{sec:tr}; see in particular \eqref{eq:almostcontact}. 
\qed\end{remark}

The next corollary follows immediately from the definitions and Lemma \ref{lem:pullback}.

%
%
\begin{corollary}\label{cor:almostcontact}
Suppose that $X$ and $Y$ are complex manifolds of dimension $2n+1$, $M$ is a closed subset of $X$,
and $F\colon X\to Y$ is a diffeomorphism which is $\dibar$-flat to order $2$ on $M$.
\begin{enumerate}[\rm (a)]
\item If $(\alpha,\beta)\in \Contf(F(M),Y)$ then $(F^*\alpha,F^*\beta)\in\Contf(M,X)$.
\item If $\alpha\in \AC(F(M),Y)$ then $F^*\alpha\in \AC(M,X)$.
\end{enumerate}
\end{corollary}

%
%

\section{Complex contact structures near totally real submanifolds}\label{sec:tr2}

In this section we prove the following parametric h-principle for complex contact structures 
along any totally real submanifold $M$ of class $\Cscr^2$ in a complex manifold $X^{2n+1}$.

%
%
\begin{theorem} \label{th:trparametric}
Let $M$ be a topologically closed totally real submanifold of class $\Cscr^2$ (possibly with boundary) 
in a complex manifold $X^{2n+1}$.  Assume that $P$ is a compact Hausdorff space and 
$Q\subset P$ is a closed subspace. Let $(\alpha_p,\beta_p) \in \Contf(X)$ $(p\in P)$ 
be a continuous family of formal complex contact structures with values in a holomorphic 
line bundle $L$ on $X$ (see Definition \ref{def:almostcontact}) such that for every $p\in Q$,
$\alpha_p\in \Conth(X)$ and $\beta_p=d\alpha_p$ on $\ker \alpha_p$. 
Then, there exist a Stein neighbourhood $\Omega\subset X$ of $M$ and a homotopy 
$(\alpha_{p,t},\beta_{p,t}) \in \Contf(X)$ $(p\in P,\ t\in [0,1])$ satisfying the following conditions.
\begin{enumerate}[\rm (a)]
\item $(\alpha_{p,0},\beta_{p,0})=(\alpha_p,\beta_p)$ for all $p\in P$.
\item The homotopy is fixed for all $p\in Q$.
\item $\alpha_{p,1}|_\Omega \in \Conth(\Omega)$ and $\beta_{p,1} = d\alpha_{p,1}$
on $\ker \alpha_1|_\Omega$ for all $p\in P$.
\end{enumerate}
\end{theorem}

This result subsumes the basic h-principle given by Theorem \ref{th:trbasic}.
The proof is based on Lemma \ref{lem:main} and the results from Sect.\ \ref{sec:AC},
along with some well known results concerning totally real submanifolds which we now recall. 

Assume that $M$ is a topologically closed totally real submanifold of class 
$\Cscr^k$ $(k\in\N)$, possibly with boundary, in a complex manifold $X$. 
Every function $f\in\Cscr^k(M)$ extends to a function $F\in \Cscr^k(X)$ which is $\Cscr^\infty$ 
smooth in $X\setminus M$ and $\dibar$-flat to order $k$ on $M$ (cf.\ \eqref{eq:dibarflatext}):
\[
	D^{k-1}(\dibar F)|_M=0.
\]
(See \cite[Lemma 4.3]{HormanderWermer1968} or \cite[Lemma 4, p.\ 148]{Boggess1991}.)
The analogous extension theorem holds for maps $f\colon M\to Y$ of class $\Cscr^k$ 
to an arbitrary complex manifold  --- such $f$ extends to a map $F:U\to Y$ on an 
open tubular Stein neighbourhood $U\subset X$ of $M$
such that $F$ is $\dibar$-flat to order $k$ on $M$.
Indeed, the graph of $f$ admits a Stein neighbourhood in $X\times Y$
according to Grauert \cite{Grauert1958AM}, so the proof reduces to the case of functions
by applying the embedding theorem for Stein manifolds into Euclidean spaces and 
the Docquier-Grauert tubular neighbourhood theorem \cite{DocquierGrauert1960}. 
(See e.g.\ \cite[proof of Corollary 3.5.6]{Forstneric2017E}.)

Let $T^\C M$ denote the complexified tangent bundle of $M$, considered as a complex
vector subbundle of $TX|_M$ of rank $m=\dim_\R M$.
The quotient bundle $\nu_M=TX|_{M}/T^\C M$ is the 
{\em complex normal bundle} of $M$ in $X$; it can be realised as a
complex vector subbundle of $TX|_M$ such that $TX|_M=T^\C M \oplus \nu_M$.
Given a diffeomorphism $f\colon M_0 \to M_1$ between totally real submanifolds
$M_0\subset X$ and $M_1\subset Y$, where $X$ and $Y$ are complex manifolds 
of the same dimension, we say that the complex normal bundles
$\pi_i\colon \nu_i\to M_i$ $(i=0,1)$ are isomorphic over $f$ if there exists
an isomorphism of complex vector bundles $\phi\colon \nu_0 \to\nu_1$
satisfying $\pi_1\circ \phi=f\circ\pi_0$. (We refer to \cite[Sect.\ 2]{ForstnericLowOvrelid2001}
for further details on this subject.)

The following result is implicitly contained in \cite[proof of Theorem 1.2]{ForstnericLowOvrelid2001}.

%
%
\begin{proposition}\label{prop:trextension}
Let $X$ and $Y$ be complex manifolds of the same dimension $n$, and let $f:M_0\to M_1$
be a diffeomorphism of class $\Cscr^k$ $(k\in \N)$ between $\Cscr^k$ totally real
submanifolds $M_0\subset X$ and $M_1\subset  Y$. If the complex normal bundles
$\pi_i\colon \nu_i\to M_i$ $(i=0,1)$ are isomorphic over $f$, then $f$ extends
to a $\Cscr^k$ diffeomorphism $F\colon U\to F(U)\subset Y$ on a neighbourhood $U\subset X$ of 
$M_0$ such that $F$ is $\dibar$-flat to order $k$ on $M$. Such extension 
always exists if $M_0$ (and hence $M_1$) is contractible, or if $M_0$ has maximal dimension $n$.
\end{proposition}

%
%
%
%
\begin{proof}[Proof of Theorem \ref{th:trparametric}]
For simplicity of exposition we consider the nonparametric case (with $P$ a singleton and 
$Q=\varnothing$); the parametric case follows by the same arguments.

We proceed in two steps. In the first step, we deform the given formal contact structure
to one that is almost contact on $M$ (see Definition \ref{def:almostcontact}).
Here we use the h-principle furnished by Lemma \ref{lem:main} and the results in
Sect.\ \ref{sec:AC}. In the second step we approximate the almost contact form on $M$ 
by a holomorphic contact form on a neighbourhood of $M$. 

The first step is accomplished  by the following lemma. 

%
%
\begin{lemma}[H-principle for almost contact structures on totally real submanifolds] \label{lem:hprinciple-AC} 
Let $M$ be a closed totally real submanifold of class $\Cscr^2$ (possibly with boundary) 
in a complex manifold $X^{2n+1}$.  
Given $(\alpha_0,\beta_0)\in \Contf(X)$, there is a homotopy $(\alpha_t,\beta_t)\in \Contf(M,X)$
$(t\in[0,1])$ such that $(\alpha_0,\beta_0)$ is the given initial pair,
$\alpha_1\in \AC(M,X)$, and $\beta_1=d\alpha_1=\di \alpha_1$ on $(\ker d\alpha_1)|_M$.
If $M$ has nonempty piecewise $\Cscr^1$ boundary $bM$ and we have
$\alpha_0|_{bM} \in \AC(bM,X)$ and $\beta_0=d\alpha_0$ on  $(\ker d\alpha_0)|_{bM}$, then 
the homotopy $(\alpha_t,\beta_t)$ may be chosen fixed on $bM$.
The analogous result holds in the parametric case.
\end{lemma}

Assume for a moment that Lemma \ref{lem:hprinciple-AC} holds and let us
complete the proof of Theorem \ref{th:trparametric}.
In view of Remark \ref{rem:observe},
there is a neighbourhood $U\subset X$ of $M$ such that 
$(\alpha_t^{1,0},\beta_t^{2,0})\in \Contf(U)$ for $t\in[0,1]$.
Hence, we may assume that $\alpha_t=\alpha_t^{1,0}$ and $\beta_t=\beta_t^{2,0}$
in $U$. By the hypothesis we also have $\alpha_1\in \AC(M,X)$ and  
$\beta_1=\di \alpha_1$ on $(\ker\alpha_1)|_M$. By a homotopic deformation (shrinking $U$ if necessary) 
we may assume that $\beta_1=\di \alpha_1$ on $(\ker\alpha_1)|_U$.  

In the next step, we find a smaller neighbourhood $U' \subset U$ of $M$ and a homotopy
in $\Contf(U')$ from $(\alpha_1,\di \alpha_1)$ to $(\wt\alpha,d\wt\alpha)$ where 
$\wt\alpha\in \Conth(U')$. This can be done by approximating $\alpha_1$ sufficiently closely in 
the fine $\Cscr^1$ topology on $M$ by a holomorphic $1$-form $\wt\alpha$ defined in a 
neighbourhood of $M$ and setting
\[
	\wt\alpha_t=(1-t)\alpha_1 + t\wt\alpha, \qquad 
	\wt\beta_t = \di \wt\alpha_t = (1-t)\di \alpha_1 + t d\wt\alpha\ \ \text{on}\ \ \ker \wt\alpha_t 
\]
for $t\in[0,1]$. Holomorphic approximation results for functions in the fine topology on 
totally real manifolds are well known, see for instance 
Manne, \O{v}relid and Wold \cite{ManneWoldOvrelid2011}
and the survey \cite{FFW2020}. These results also apply to sections 
of holomorphic vector bundles as shown in  \cite[proof of Theorem 2.8.4]{Forstneric2017E}.
Finally, the homotopy in $\Contf(U')$ from $(\alpha_0,\beta_0)$ to 
$(\wt \alpha,d\wt\alpha)$, constructed above, can be extended to all of $X$ in a standard way 
by using a cut-off function on $X$ in the parameter  of the homotopy,
thereby yielding a homotopy in $\Contf(X)$ which equals the given one 
on a smaller Stein neighbourhood $\Omega \subset U'$ of $M$ and it agrees with 
$(\alpha_0,\beta_0)$ on $X\setminus U'$.

Assuming that Lemma \ref{lem:hprinciple-AC} holds, this completes the proof 
of Theorem \ref{th:trparametric}. As said before, the parametric case follows the same pattern 
and we omit the details. 
\end{proof}

\begin{proof}[Proof of Lemma \ref{lem:hprinciple-AC}]
Choose a triangulation of $M$ and let $M_k$ denote its $k$-dimensional skeleton, i.e.,
the union of all cells of dimension at most $k$.
Assume inductively that for some $k<m=\dim M$ we have already found a homotopy
in $\Contf(M,X)$ from $(\alpha_0,\beta_0)$ to $(\alpha,\beta)\in \Contf(M,X)$ satisfying
the following conditions:
\[
	\alpha\in \AC(M_k,X),\quad \beta=d\alpha \ \ \text{on}\ (\ker \alpha)|_{M_k},\quad
	\alpha\wedge (d\alpha)^n|_{M_k} \ne 0.
\]
The inductive step amounts to deforming $(\alpha,\beta)$ by a homotopy in $\Contf(M,X)$
that is fixed on $M_k$ to another pair $(\wt \alpha,\wt \beta)\in \Contf(M,X)$ such that 
\[
	\wt\alpha\in \AC(M_{k+1},X),\quad \wt \beta=d\wt\alpha\ \ \text{on}\ (\ker \wt \alpha)|_{M_{k+1}},
	\quad \wt\alpha\wedge (d\wt\alpha)^n|_{M_{k+1}} \ne 0.
\]
This can be done by applying Lemma \ref{lem:main} successively on each $(k+1)$-dimensional
cell $C^{k+1}$ in the given triangulation of $M$; we now explain the details. 

Let $L\to X$ be the holomorphic line bundle such that $\alpha_0,\beta_0$ have values in $L$.
Note that $L$ is holomorphically trivial over a neighbourhood of the cell $C^{k+1}$
by the Oka-Grauert principle, so we may consider all 
our $L$-valued differential forms to be scalar-valued there.
The cell $C^{k+1}$ is diffeomorphic
to a compact contractible domain $D^{k+1}\subset \R^{k+1}$ as in Lemma \ref{lem:main}.
We identify $\R^{k+1}$ with $\R^{k+1}\times \{0\}^{2n-k} \subset \R^{2n+1}\subset \C^{2n+1}$.
Since $M$ is totally real and of class $\Cscr^2$, any diffeomorphism 
$F\colon C^{k+1} \to D^{k+1}$ of class $\Cscr^2$ extends to a diffeomorphism $F$
from a neighbourhood of $C^{k+1}$ in $X$ onto a neighbourhood of $D^{k+1}$ in $\C^{2n+1}$
which is $\dibar$-flat to order $2$ on $C^{k+1}$ (see Proposition \ref{prop:trextension}). 
The inverse $G=F^{-1}$ is then $\dibar$-flat to order $2$ on $D^{k+1}$. 
By Corollary \ref{cor:almostcontact} we have that
\begin{enumerate}[\rm (i)]
\item $(G^*\alpha,G^*\beta)\in \Contf(D_{k+1},\C^{2n+1})$, 
\item $G^*\alpha \in \AC(bD^{k+1},\C^{2n+1})$, and
\item $G^*\beta=d (G^*\alpha)$ holds on $\ker(G^*\alpha)$ at all points of $bD^{k+1}$.
\end{enumerate}
By Lemma \ref{lem:main} we can deform $(G^*\alpha,G^*\beta)$ 
by a homotopy in $\Contf(D^{k+1},\C^{2n+1})$ that is fixed on $bD^{k+1}$ to an element 
$(\alpha',\beta')\in \Contf(D^{k+1},\C^{2n+1})$ such that $\alpha'\in \AC(D^{k+1},\C^{2n+1})$ and 
$\beta'=d\alpha'$ on $(\ker \alpha')|_{D_{k+1}}$. 
(Lemma  \ref{lem:main} applies verbatim if $k+1=m=2n+1$. If $k+1<2n+1$, 
we can apply it on $D^{k+1}\times r\D^{2n-k}$ for some $r>0$, 
where $\D^{2n-k}$ is the closed ball around the origin in $\R^{2n-k}$.
We can extend $G^*\alpha$ to an element of $\AH^1(D^{k+1}\times r\D^{2n-k},\C^{2n+1})$
whose restriction to $bD^{k+1}\times r\D^{2n-k}$ belongs to 
$\AC(bD^{k+1}\times r\D^{2n-k},\C^{2n+1})$ and apply Lemma  \ref{lem:main} to this extension.)
By Corollary \ref{cor:almostcontact} we have $F^*\alpha' \in \AC(C^{k+1},X)$ and
$d(F^*\alpha')=F^*\beta'$ on $\ker(F^*\alpha')$ along $C^{k+1}$. We also use $F^*$ to transfer
the homotopy in $\Contf(D_{k+1},\C^{2n+1})$, connecting 
$(G^*\alpha,G^*\beta)$ to $(\alpha',\beta')$, to a homotopy in $\Contf(C^{k+1},X)$ 
which is fixed on $bC^{k+1}$ and connects $(\alpha,\beta)$ to $(F^*\alpha',F^*\beta')$. 

This completes the basic induction step. 
Applying this procedure successively on each $(k+1)$-cell in the given triangulation 
of $M$ yields a desired  almost complex structure $\wt \alpha \in \AC(M_{k+1},X)$. 
In the final step when $k+1=m$ we obtain an element $\alpha_1\in \AC(M,X)$.

Clearly all steps can be carried out with a continuous dependence on a parameter,
and by using cut-off functions on the parameter space we can ensure that the 
homotopy is fixed for the parameter values $p\in Q$. 
This yields the corresponding parametric h-principle.
\end{proof}

%
%

\section{Extending a complex contact structure across a totally real handle}
\label{sec:handle}

Recall that a compact set in a complex manifold $X$ is called a {\em Stein compact}
if it admits a basis of open Stein neighbourhoods in $X$.
The following lemma provides a key induction step in the proof of Theorems
\ref{th:basic}, \ref{th:ifthen}, and \ref{th:parametric}.

%
%

\begin{lemma}\label{lem:handle}
Let $K$ and $S=K\cup M$ be Stein compacts in a complex manifold $X^{2n+1}$, where
$M=\overline {S\setminus K}$ is an embedded totally real submanifold of class $\Cscr^2$.
Let $(\alpha,\beta)\in\Contf(X)$ be a formal contact structure with values in 
a holomorphic line bundle $L$ on $X$. Assume that there is an open neighbourhood 
$U\subset X$ of $K$ such that $\alpha|_U\in\Conth(U)$ and $\beta=d\alpha$
on $\ker\alpha|_U$. Then, there exist a neighbourhood $\Omega_0\subset U$ of $K$,
a Stein neighbourhood $\Omega\subset X$ of $S$,
and a homotopy $(\alpha_t,\beta_t)\in \Contf(X)$ $(t\in[0,1])$ satisfying the
following conditions.
\begin{enumerate}[\rm (i)]
\item $(\alpha_0,\beta_0)=(\alpha,\beta)$ on $\Omega_0$.
\item $\alpha_t|_{\Omega_0}\in\Conth(\Omega_0)$ and 
$\beta_t=d\alpha_t$ on $\ker\alpha_t|_{\Omega_0}$ for all $t\in[0,1]$.
\item $\alpha_t$ approximates $\alpha$ as closely as desired uniformly on $K$ and
uniformly in $t\in [0,1]$.
\item $\alpha_1|_\Omega\in \Conth(\Omega)$ and $\beta_1=d\alpha_1$ on $\ker\alpha_1|_\Omega$.
\end{enumerate}
The analogous result holds for a continuous family $\{(\alpha_p,\beta_p)\}_{p\in P}\subset \Contf(X)$
where $P$  is a compact Hausdorff space; the homotopy may be kept fixed for the parameter values 
in a closed subset $Q\subset P$ such that $\alpha_p\in\Conth(X)$ for all $p\in Q$.
\end{lemma}

\vspace{-5mm}

\begin{proof}
Let $U\subset X$ be a relatively compact neighbourhood of $K$ as in the statement of the lemma;
in particular, $\alpha|_U\in \Conth(U)$.
Choose a closed domain $M_0\subset M$ with $\Cscr^2$ boundary such that $M_0\cap K=\varnothing$
and $K':=K\cup \overline{M\setminus M_0}\subset U$. By Lemma \ref{lem:hprinciple-AC} we can deform 
$(\alpha,\beta)$ through a family of formal contact structures $(\alpha_t,\beta_t) \in \Contf(X)$  
such that the deformation is fixed on a neighbourhood of $K'$, 
and at $t=1$ we have that  $\alpha_1|_{M_0}\in\AC(M_0,X)$ and $\beta_1=d\alpha_1$
on $(\ker\alpha_1)|_{M_0}$. Note that $\alpha_1$ is holomorphic on a neighbourhood of $K'$ 
(where it equals $\alpha_0$) and is asymptotically holomorphic along $M$.

By the Mergelyan approximation theorem,  
we can approximate $\alpha_1$ and its $1$-jet along $M$ 
as closely as desired in the $\Cscr^1$ topology on $S=K\cup M$ by an $L$-valued 
holomorphic $1$-form $\wt \alpha_1$ defined on a neighbourhood of $S$.
We refer to \cite[Theorem 20]{FFW2020} for the relevant version of Mergelyan's theorem.
(In the cited source the reader can also find references to the previous works;
see in particular Manne, \O{v}relid and Wold \cite{ManneWoldOvrelid2011}.
The proof of \cite[Theorem 20]{FFW2020} easily adapts to provide jet-approximation;
see Chenoweth \cite[Proposition 7]{Chenoweth2019}.
Although the cited results are stated for functions, they also hold for 
sections of holomorphic vector bundles over Stein domains 
as shown in \cite[proof of Theorem 2.8.4]{Forstneric2017E}.)
If the approximation of $\alpha_1$ by $\wt\alpha_1$ is close enough on $S$, the family
$(1-t)\alpha_1+t\wt\alpha_1$ $(t\in [0,1])$ is a homotopy of 
holomorphic contact forms on a neighbourhood of $K'$, and its restriction to $M_0$ is a homotopy in
the space $AC(M_0,X)$ of almost contact forms on $M_0$. 

By combining the homotopies from these two steps, we get a homotopy $(\alpha_t,\beta_t)$
on a neighbourhood $V\subset X$ of $S=K\cup M$ satisfying the conclusion of the lemma.
Finally, by inserting a smooth cutoff function on $X$ into the parameter of the homotopy,
we can glue the resulting homotopy with $(\alpha_0,\beta_0)=(\alpha,\beta)$ outside a 
Stein neighbourhood $\Omega\subset V$ of $S$.

It is clear that the same proof applies in the parametric situation. 
The main ingredients are the parametric version of Lemma \ref{lem:main} and a 
parametric version of Mergelyan's theorem from \cite[Theorem 20]{FFW2020}.
The latter is easily obtained from the basic (nonparametric) case 
by applying a continuous partition of unity on the parameter space.
(Compare with the proof of the parametric Oka-Weil theorem
in \cite[Theorem 2.8.4]{Forstneric2017E}.)
\end{proof}

%
%

\section{Proofs of the main results} \label{sec:proofs}

\begin{proof}[Proof of Theorem \ref{th:basic}]
We follow the scheme explained in the paper \cite{ForstnericSlapar2007MZ} by
Slapar and the author; see in particular the proof of Theorem 1.2 in the cited source. 
Complete expositions of this construction can also be found in 
\cite[Chap.\ 8]{CieliebakEliashberg2012} and \cite[Secs.\ 10.9--10.11]{Forstneric2017E}.

Choose a smooth strongly plurisubharmonic Morse exhaustion function $\rho\colon X\to \R_+$. 
Let $p_0,p_1,p_2,\ldots\in X$ be the critical points of $\rho$ with
$\rho(p_0)<\rho(p_1)<\cdots$; thus $p_0$ is a minimum of $\rho$. 
Choose numbers $c_j\in\R$ satisfying 
\[
	\rho(p_0) < c_0 < \rho(p_1)< c_1 < \rho(p_2) < c_2 <\ldots.
\]
For each $j=0,1,\ldots$ we set $X_j=\{x\in X\colon \rho(x) < c_j\}$.
Note that $\rho$ has a unique critical point $p_j$ in $X_j \setminus X_{j-1}$ for each $j=1,2,\ldots$.
(If $\rho$ has only finitely many critical points $p_0,\ldots, p_m$, the process described in the sequel
will stop after $m+1$ steps and the domain $X_m=\{\rho <c_m\}$ is diffeotopic to $X$. This is always
the case if $X$ is an affine algebraic manifold.)
By choosing the number $c_0$ close enough to $\rho(p_0)$, we can arrange by a homotopy 
in $\Contf(X)$ that $\alpha_0$ is a holomorphic contact form on a neighbourhood of the set 
$\overline X_0 = \{\rho \le c_0\}$ and $\beta_0=d\alpha_0$ on $\ker \alpha_0$ holds. 
 
Fix a number $\epsilon>0$. We shall inductively construct the following objects:
\begin{enumerate}[\rm (a)]
\item an increasing sequence of relatively compact, 
smoothly bounded, strongly pseudoconvex domains 
$W_0\subset W_1\subset W_2\subset \cdots$ in $X$, with $W_0=X_0$,
\item a sequence of formal contact structures $(\alpha_j,\beta_j)\in \Contf(X)$ $(j=1,2,\ldots)$
with values in the given holomorphic line bundle $L\to X$, and 
\item a sequence of smooth diffeomorphisms $h_j\colon X\to X$ $(j=0,1,\ldots)$ with $h_0=\Id_X$,
\end{enumerate}
satisfying the following conditions for all $j=1,2,\ldots$.
\begin{enumerate}[\rm (i)]
\item  The set $\overline W_{j-1}$ is $\Oscr(W_j)$-convex.
\item  There is an open neighbourhood $U_j\subset X$ of  $\overline W_j$
such that $\alpha_j|_{U_j}\in \Conth(U_j)$  and $d\alpha_j=\beta_j$ on $\ker\alpha_j|_{U_j}$. 
(This already holds for $j=0$.)
\item There is a homotopy $(\alpha_{j,t},\beta_{j,t})\in \Contf(X)$ 
$(t\in [0,1])$ such that $(\alpha_{j,0},\beta_{j,0})=(\alpha_{j-1},\beta_{j-1})$,
$(\alpha_{j,1},\beta_{j,1})=(\alpha_j,\beta_j)$, and for every $t\in [0,1]$,
$\alpha_{j,t}$ is a holomorphic contact form in a
neighbourhood of $\overline W_{j-1}$ with $d\alpha_{j,t}=\beta_{j,t}$ on $\ker\alpha_{j,t}$ there. 
\item  $\sup_{x\in W_{j-1}} |\alpha_{j,t}(x) - \alpha_{j-1}(x)| < \epsilon 2^{-j}$, 
where the difference of forms is measured with respect to a fixed pair of
hermitian metrics on $T^*X$ and  $L$.
\item  $h_j(X_j)=W_j$ and $h_j=\Id_X$ on $X\setminus X_{j+1}$ (hence,
$h_j(X_{j+1})=X_{j+1}$).
\item $h_j=g_j\circ h_{j-1}$ where $g_j\colon X\to X$ is a diffeomorphism which 
maps $X_j$ onto $W_j$ and is diffeotopic to $\Id_X$ by a diffeotopy 
that equals $\Id_X$ on $\overline W_{j-1}\cup (X\setminus X_{j+1})$.
\end{enumerate}

Granted such sequences, the domain $\Omega=\bigcup_{j} W_j \subset X$
is Stein in view of condition (i), the limit $\wt \alpha = \lim_{j\to \infty} \alpha_j$ exists 
and is a holomorphic contact form  on $\Omega$ in view of (ii) and (iv), 
and the individual homotopies in (iii) can be put together into a homotopy in 
$\Contf(\Omega)$ from $(\alpha_0,\beta_0)$ to $(\wt\alpha,d\wt\alpha)$
(see conditions (iii) and (iv)). Furthermore, conditions (v) and (vi) ensure that 
the sequence $h_j$ converges to a diffeomorphism $h=\lim_{j\to\infty} h_j \colon X\to \Omega$
satisfying the conclusion of Theorem \ref{th:basic}. 
With a bit more care in the choice of $W_j$ at each step, we can ensure that $\Omega$ is
smoothly bounded and strongly pseudoconvex. In general we cannot 
choose $\Omega$ to be relatively compact, unless $X$ 
admits an exhaustion function $\rho\colon X\to\R$ with at most 
finitely many critical points. In the latter case, the above process clearly terminates
in finitely many steps and yields a holomorphic contact form on a
bounded strongly pseudoconvex domain $\Omega\Subset X$ diffeotopic to $X$. 

We now describe the induction step. To the strongly pseudoconvex domain $W_{j-1}$ 
we attach the disc $M_j := h_{j-1}(D_j)$, where $D_j\subset X_j \setminus X_{j-1}$ 
(with $bD_j\subset bX_{j-1}$) is the unstable disc at the critical point 
$p_j \in X_j \setminus X_{j-1}$.
By \cite[Lemma 3.1]{ForstnericSlapar2007MZ} we can isotopically deform
$M_j$ to a smooth totally real disc in $X\setminus W_{j-1}$ 
attached to $bW_{j-1}$ along the Legendrian sphere $bM_j\subset bW_{j-1}$. 
Lemma \ref{lem:handle} provides the next element $(\alpha_j,\beta_j)\in\Contf(X)$, 
and a homotopy  $(\alpha_{j,t},\beta_{j,t})\in \Contf(X)$ $(t\in[0,1])$ 
satisfying condition (iii), such that $\alpha_j$ is a holomorphic contact form on a 
strongly pseudoconvex handlebody $W_j\supset \overline W_{j-1}\cup M_j$ 
and $\beta_j=d\alpha_j$ there. The next diffeomorphism $h_j=g_j\circ h_{j-1}$
satisfying conditions (v) and (vi) is then furnished by Morse theory. This concludes the proof.

Conditions (v) and (vi) show that the domain $\Omega$ is diffeotopic to $X$.
By a more precise argument in the induction step one can also ensure
the existence a diffeotopy $h_t:X\to h_t(X)\subset X$ from $h_0=\Id_X$ 
to a diffeomorphism $h_1=h:X\to \Omega$ through a family of Stein domains $h_t(X)\subset X$; 
see \cite[Theorem 8.43 and Remark 8.44]{CieliebakEliashberg2012}. This depends on the
stronger technical result given by \cite[Theorem 8.5, p.\ 157]{CieliebakEliashberg2012}.
\end{proof}

%
%
%
%

The same proof gives the following parametric extension of Theorem \ref{th:basic}.

\begin{theorem}\label{th:parametric}
Assume that $X$ is a Stein manifold of dimension $2n+1\ge 3$ and $Q\subset P$ are compact Hausdorff spaces.
Let $(\alpha_p,\beta_p)\in \Contf(X)$ be a continuous family of formal contact structures such that for every 
$p\in Q$, $(\alpha_p,\beta_p=d\alpha_p)$ is a holomorphic contact structure. 
Then there are a Stein domain $\Omega\subset X$ diffeotopic to $X$
and a homotopy $(\alpha_{p,t},\beta_{p,t})\in  \Contf(X)$ $(p\in P,\ t\in [0,1])$ 
which is fixed for all $p\in Q$ such that 
$(\alpha_{p,1},\beta_{p,1}=(d\alpha_{p,1})_{\ker {\alpha_{p,1}}})$ 
is a holomorphic contact structure on $\Omega$ for every $p\in P$.
\end{theorem}

To see this, we follow the proof of Theorem \ref{th:basic} and note that, 
in the inductive step, the domain $W_j$ (a smoothly bounded 
tubular Stein neighbourhood of $\overline W_{j-1}  \cup M_j$) 
can be chosen such that Lemma \ref{lem:handle} provides the next family 
$\{(\alpha_{p,j},\beta_{p,j})\}_{p\in P}\in\Contf(X)$ satisfying condition (iii), 
where $\alpha_{p,j}$ is a holomorphic contact form on $W_j$ and 
$\beta_{p,j}=d\alpha_{p,j}$ on $W_j$ for all $p\in P$. 

We recall the following definition \cite[Definition 5.7.1]{Forstneric2017E}.

%
%
\begin{definition}\label{Cartan-pair}
A pair $(A,B)$ of compact subsets in a complex manifold $X$
is a {\em Cartan pair} if it satisfies the following two conditions:   
\begin{enumerate}[\rm(i)]  
\item $A$, $B$, $C=A\cap B$, and $D=A\cup B$ are Stein compacts
(i.e., they admit a basis of open Stein neighbourhoods in $X$), and
\item $A,B$ are {\em separated} in the sense that
$\overline{A\setminus B}\cap \overline{B\setminus A} =\varnothing$.
\end{enumerate}
\end{definition}

A particularly simple kind of a Cartan pair is a {\em convex bump}; see
\cite[Definition 5.10.2]{Forstneric2017E}. This means that, in addition to the conditions 
in Definition \ref{Cartan-pair}, there is a coordinate neighbourhood 
$(U,z)$ of $B$ in $X$, with a biholomorphic map $z:U\to\wt U\subset\C^n$
$(n=\dim X)$, such that $z(B)$ and $z(C)=z(A\cap B)$ are compact convex sets in $\C^n$. 

In the proof of Theorem \ref{th:ifthen} we shall need the following 
gluing lemma for holomorphic contact forms on Cartan pairs.
(The analogous gluing lemma for nonsingular holomorphic foliations given by exact 
holomorphic $1$-forms is \cite[Theorem 4.1]{Forstneric2003AM}.)

%
%
\begin{lemma}[Gluing lemma for holomorphic contact forms] \label{lem:gluingcontact}
Let $(A,B)$ be a Cartan pair in a complex manifold $X^{2n+1}$.
Assume that $\alpha,\beta$ are holomorphic contact forms on open neighbourhoods
of $A$ and $B$, respectively. If $\beta$ is sufficiently uniformly close to $\alpha$ 
on a fixed neighbourhood of $C=A\cap B$, then there exists a holomorphic 
contact form $\wt\alpha$ on a neighbourhood of $A\cup B$ 
which approximates $\alpha$ uniformly on $A$ and
approximates $\beta$ uniformly on $B$.
\end{lemma}

\begin{proof}
Let $\alpha$ and $\beta$ be holomorphic contact forms on open neighbourhoods
$A'\supset A$ and $B'\supset B$, respectively. Set $C'=A'\cap B'$ and define 
\[
	\alpha_t=(1-t)\alpha + t\beta\ \ \text{in $C'$\ \ for $t\in[0,1]$}.
\] 
Assuming that $\beta$ is sufficiently uniformly close to $\alpha$ on $C'$,
$\alpha_t$ is a contact form on a smaller neighbourhood of $C=A\cap B$ for every $t\in[0,1]$.
By the proof of Gray's stability theorem (see \cite{Gray1959AM} 
or \cite[p. 60]{Geiges2008} for the smooth case) we find 
\begin{enumerate}
\item a neighbourhood $C''\subset C'$ of $C$,
\item an isotopy of biholomorphic maps $\phi_t\colon C''\to \phi_t(C'') \subset C'$ $(t\in[0,1])$ 
with $\phi_0=\Id$ and $\phi_t$ close to the identity for all $t\in[0,1]$, and
\item a family of nowhere vanishing holomorphic functions $\lambda_t\colon C''\to\C^*$ 
close to $1$, with $\lambda_0=1$,
\end{enumerate} 
satisfying $\phi_t^*\alpha_t=\lambda_t \alpha$ on $C''$ for every $t\in[0,1]$.
In particular, we have 
\[
	\text{$\phi_1^*\beta=\lambda_1 \alpha$\ \ on $C''$.}
\]
Assuming that $\phi_1$ is sufficiently uniformly close to the identity on $C''$
(which holds if $\beta$ is close enough to $\alpha$ on $C'$), 
we can apply the splitting lemma \cite[Theorem 9.7.1]{Forstneric2017E} to obtain 
\[
	\phi_1\circ \phi_A = \phi_B
\]
on a neighbourhood of $C$, where $\phi_A$ and $\phi_B$ are biholomorphic maps 
close to the identity on open neighbourhoods of $A$ and $B$, respectively. 
On a neighbourhood of $C$ we then have 
\[
	(\lambda_1\circ \phi_A) \,\cdotp \phi_A^* \alpha
	= \phi_A^* (\lambda_1\alpha) 
	= \phi_A^* (\phi_1^* \beta) = (\phi_1 \circ \phi_A)^* \beta = 
	\phi^*_B \beta.
\] 
This shows that the holomorphic contact forms $\phi_A^* \alpha,\phi^*_B \beta$, defined on
neighbourhoods of $A$ and $B$, respectively, have the same kernel on a neighbourhood of $C$,
and hence they define a holomorphic contact structure $\tilde \xi$
on a neighbourhood of $A\cup B$. Assuming as we may that the function $\lambda_1\circ \phi_A$
is sufficiently close to $1$ on a neighbourhood of $C$, we can solve 
a multiplicative Cousin problem on the Cartan pair $(A,B)$ and correct the above $1$-forms 
by the respective factors to obtain a holomorphic $1$-form $\wt \alpha$ on a
neighbourhood of $A\cup B$, with $\ker\wt\alpha=\tilde\xi$,  
which approximates $\alpha$ and $\beta$ on $A$ and $B$, respectively.
\end{proof}

%
%
%
\begin{proof}[Proof of Theorem \ref{th:ifthen}]
We follow the inductive scheme used in Oka theory;
see for instance \cite[the proof of Theorem 5.4.4]{Forstneric2017E}.

We use the notation established in the proof of Theorem \ref{th:basic}. 
The only difference from that proof is that we can now extend a holomorphic contact form 
(by approximation) from a neighbourhood of the sublevel set $\overline X_{j-1}=\{\rho\le c_{j-1}\}$
to a neighbourhood of $\overline X_j=\{\rho\le c_j\}$, provided it extends as a formal
contact structure. 

The first step, namely the extension to a Stein handlebody $W_{j-1}$
around $\overline X_{j-1}\cup M_{j}$ (where $M_j$ is a totally real disc
which provides the change of topology at the critical point $p_j\in X_j\setminus X_{j-1}$)
is furnished by the proof of Theorem \ref{th:basic}. We may arrange the process so that 
$X_j$ is a noncritical strongly pseudoconvex extension of $W_{j-1}$
(see \cite[Sect.\ 5.10]{Forstneric2017E}). This implies that we can obtain $X_j$
from $W_{j-1}$ by attaching finitely many convex bumps
(see \cite[Lemma 5.10.3]{Forstneric2017E} for the details).
We now successively extend the contact form (by approximation) across each bump. 
At every step of this process we have a Cartan pair $(A,B)$, 
where $B$ is a convex bump attached to a compact 
strongly pseudoconvex domain $A$ along the set $C=A\cap B$.
(The sets $C\subset B$ are convex in some holomorphic coordinates on a neighbourhood of $B$ in $X$.) 
We also have a holomorphic contact form $\alpha$ on a neighbourhood of $A$. 
Assuming that Problem \ref{pr:Runge} has an affirmative answer, 
we can approximate $\alpha$ uniformly on a neighbourhood of $C$
by a holomorphic contact form $\beta$ on a neighbourhood of $B$.
If the approximation is close enough, Lemma \ref{lem:gluingcontact}
furnishes a holomorphic contact form $\wt\alpha$ on neighbourhood
of $A\cup B$ which approximates $\alpha$ uniformly on $A$.
In finitely many steps of this kind we approximate  the given holomorphic 
contact form on $\overline W_{j-1}$ by a holomorphic contact form on a neighbourhood of $\overline X_j$. 
Hence, this process converges to a holomorphic contact form on all of $X$.
The same holds in the parametric case if the parametric
version of Problem \ref{pr:Runge} has an affirmative answer.
\end{proof}

\subsection*{Acknowledgements}
The author is supported by the research program P1-0291 and grants J1-7256 and
J1-9104 from ARRS, Republic of Slovenia. He wishes to thank Yakov Eliashberg,
Finnur L\'arusson, Marko Slapar, and Jaka Smrekar for helpful discussions.




\vspace*{5mm}
\noindent Franc Forstneri\v c

\noindent Faculty of Mathematics and Physics, University of Ljubljana, Jadranska 19, SI--1000 Ljubljana, Slovenia

\noindent Institute of Mathematics, Physics and Mechanics, Jadranska 19, SI--1000 Ljubljana, Slovenia

\noindent e-mail: {\tt franc.forstneric@fmf.uni-lj.si}

\end{document}